\documentclass[12pt]{amsart}
\usepackage{amsfonts}
\usepackage{amsmath}
\usepackage{amssymb}
\usepackage{mathrsfs}
\usepackage{graphicx}
\usepackage{enumerate}
\usepackage{multicol}
\usepackage{a4wide}
\usepackage[all,cmtip]{xy}

\newtheorem{theorem}{Theorem}
\newtheorem{corollary}{Corollary}
\newtheorem{lemma}{Lemma}
\newtheorem{definition}{Definition}
\newtheorem{proposition}{Proposition}
\theoremstyle{remark}
\newtheorem{remark}{Remark}
\newtheorem{example}{Example}

\DeclareMathOperator{\id}{\it id}
\DeclareMathOperator{\pr}{pr}

\DeclareMathOperator{\spn}{span}
\DeclareMathOperator{\sgn}{sgn }

\DeclareMathOperator{\Arg}{Arg}
\DeclareMathOperator{\ad}{ad}
\DeclareMathOperator{\Ad}{Ad}

\DeclareMathOperator{\SU}{SU}
\DeclareMathOperator{\wSU}{ \widetilde{SU}}

\DeclareMathOperator{\Diff}{Diff}
\DeclareMathOperator{\Vect}{Vect}

\DeclareMathOperator{\su}{\mathfrak{su}}

\DeclareMathOperator{\Lieg}{\mathfrak{g}}
\DeclareMathOperator{\Lieh}{\mathfrak{h}}
\DeclareMathOperator{\Liep}{\mathfrak{p}}
\DeclareMathOperator{\Liek}{\mathfrak{k}}

\DeclareMathOperator{\comp}{{\mathbb{C}}}
\DeclareMathOperator{\integer}{{\mathbb{Z}}}
\DeclareMathOperator{\real}{\mathbb{R}}
\DeclareMathOperator{\unitD}{\mathbb{D}}

\DeclareMathOperator{\calF}{\mathcal{F}}
\DeclareMathOperator{\calG}{\mathcal{G}}
\DeclareMathOperator{\calH}{\mathcal{H}}
\DeclareMathOperator{\calA}{\mathcal{A}}

\DeclareMathOperator{\calE}{\mathcal{E}}

\DeclareMathOperator{\calV}{\mathcal{V}}
\DeclareMathOperator{\calJ}{\mathcal{J}}

\DeclareMathOperator{\scrC}{\mathscr{C}}
\DeclareMathOperator{\scrS}{\mathscr{S}}
\DeclareMathOperator{\scrT}{\mathscr{T}}

\DeclareMathOperator{\Ann}{Ann}
\DeclareMathOperator{\Var}{Var}
\DeclareMathOperator{\End}{end}
\DeclareMathOperator{\abscon}{AC}
\DeclareMathOperator{\Lie}{Lie}

\DeclareMathOperator{\Rot}{Rot}
\DeclareMathOperator{\wDiff}{\widetilde{Diff}}
\DeclareMathOperator{\llangle}{\langle \! \langle}
\DeclareMathOperator{\rrangle}{\rangle \! \rangle}

\title[Sub-Riemannian geometry...]{Sub-Riemannian geometry on infinite-dimensional manifolds}
\author[E. Grong, I. Markina, A. Vasil'ev]{Erlend Grong, Irina Markina, and Alexander Vasil'ev}

\address{ 
\noindent \newline Department of Mathematics\newline
University of Bergen \newline P.O.~Box~7803 \newline Bergen N-5020, Norway \bigskip}

\email{\newline erlend.grong@math.uib.no \newline irina.markina@math.uib.no \newline alexander.vasiliev@math.uib.no}

\thanks{The authors have been  supported by the grant of the Norwegian Research Council \#204726/V30, by the NordForsk network `Analysis and Applications', grant \#080151, and by the European Science Foundation Research Networking Programme HCAA. This work was completed while the authors were visiting Mittag-Leffler institute, Sweden in the Fall 2011.}


\begin{document}

\begin{abstract}
We generalize the concept of sub-Riemannian geometry to infinite-dimensional manifolds modeled on convenient vector spaces. On a sub-Riemannian manifold $M$, the metric is  defined only on a sub-bundle $\calH$ of the tangent bundle $TM$, called the horizontal distribution. Similarly to the finite-dimensional case, we are able to split possible candidates for minimizing curves into two categories: semi-rigid curves that  depend only on $\calH$, and normal geodesics that depend  both on $\calH$ itself  and on the metric on $\calH$. In this sense, semi-rigid curves in the infinite-dimensional case generalize the notion of singular curves for finite dimensions. In particular, we study  the case of regular Lie groups. As examples, we consider  the group of sense-preserving diffeomorphisms $\Diff S^1$ of the unit circle and the Virasoro-Bott group with their respective horizontal distributions chosen to be the Ehresmann connections with respect to a projection to the space of normalized univalent functions. In these cases we prove controllability and find formulas for the normal geodesics with respect to the pullback of the invariant K\"ahlerian metric on the class of normalized univalent functions. The geodesic equations are analogues to the Camassa-Holm, Huter-Saxton, KdV, and other known non-linear PDE.
\end{abstract}

\maketitle

\section{Introduction}

The main goal of the paper is to study the geometry of infinite-dimensional manifolds with non-holonomic constraints, which is a generalization of sub-Riemannian geometry in the finite-dimensional case. A sub-Riemannian manifold is a triple $(M, \calH, \mathbf h)$, such that $M$ is a connected smooth manifold, $\calH$ is a smooth sub-bundle of $TM$, and $\mathbf h$ is a Riemannian metric on $\calH$. The co-dimension of $\calH$ is assumed to be positive, otherwise we consider a standard Riemannian manifold.

Sub-Riemannian geometry on finite-dimensional manifolds is well studied, and has been proved to have important applications in many areas ranging from optimal control theory \cite{Brockett} and sub-elliptic operators \cite{ABGR,BB} to mathematical physics \cite{Giannoni}.  Typical general references are~\cite{AS, Mon, Str1, Str2}. Unlike the standard Riemannian geometry on  $M$, the metric is  defined only on a sub-bundle $\calH$ of the tangent bundle $TM$. The distance between two fixed points is measured in terms of the length of the curves connecting them and passing tangentially to $\calH$ at any point. Such curves are called {\it horizontal}. The distance is finite if every pair of points can be connected by at least one horizontal curve and is achieved on the curves  of  minimal length.
The standard way to ensure that any pair of points can be connected by a horizontal curve, is to require that $\calH$ is {\it bracket generating}. Connectivity by horizontal curves then follows from the Rashevski{\u\i}-Chow Theorem \cite{Chow, Rashevsky}. The necessary condition for  minimizing curves is given,~e.g.,  by the Pontryagin Maximum Principle \cite{AS}. This condition implies that the optimal curves are of two types: normal geodesics that behave similarly to the standard Riemannian geodesics, and singular curves that depend only  on the distribution $\calH$ itself and not on the metric on $\calH$.

We are going to generalize as much as possible of the above construction  to infinite-dimensional manifolds, where  constraints are reflected in the sub-bundle $\calH$. Unfortunately, we loose both the Rashevski{\u\i}-Chow Theorem and the Pontryagin Maximum Principle along the way, however, we still have the tools of variational calculus developed for manifolds modeled on convenient vector spaces at our disposal. Using them, we are aimed at developing an analogue of sub-Riemannian geometry in the infinite-dimensional setting.

The outline and main results of the paper are as follows. After giving motivation in the following section, we introduce basic definitions of sub-Riemannian geometry
on infinite-dimensional manifolds with splitting sub-bundles in Section~3. The notion of semi-rigid curves is introduced. Semi-rigid curves play the role analogous to singular 
curves in finite-dimensional sub-Riemannian geometry. The normal geodesics are shown to satisfy the Euler equation. In Section~4 we are focused on an important
particular case of infinite-dimensional manifolds, the regular Lie groups.
Section~5 contains applications
of the results for the concrete case of the group of orientation preserving diffeomorphisms of the unit circle $\Diff S^1$ and its central extension known as the  Virasoro-Bott group. There we construct a metric that allows us to apply the theorems about geodesics proved in previous sections and analyze the formulas of geodesics. Applying Sobolev metrics and a metric related to the K\"ahlerian structure on the space of normalized univalent functions, it turns out that the Euler equations for the geodesics recover analogues to the Burgers, KdV, Camassa-Holm, and Hunter-Saxton equations.

\section{Motivation}\label{motivation}

First, let us agree on some basic conventions and definitions.
For simplicity, most of the curves in our paper are parametrized on the unit interval $I =[0,1]$. All partial differential operators are also shortened writing $\partial_x:=\tfrac{\partial}{\partial x}$. Partial derivatives with respect to $t$ for a curve $
\gamma(t)$, $t\in I$, are usually denoted by  dot, $\partial_t \gamma = \dot \gamma.$

For a map between two manifolds $f:M \to B$, the tangent map, or the differential of this map, is written as $df:TM \to TB.$ If $\alpha$ is a form on $M$, and $v \in T_mM$, then we will write $\alpha(m)(v)$ as simply $\alpha(v)$, whenever it is clear from the context which tangent space the vector $v$ belongs to. Metrics are denoted by boldface letters, e.g., $\mathbf g, \mathbf h$.

We will work with manifolds modeled on {\it $c^{\infty}$-open subsets of convenient vector spaces} following the terminology  found in \cite{KrieglMichor}. A convenient vector space is a locally convex vector space, where the most general notion of smoothness, based on the notion of smooth curves, is introduced and the vector space satisfies the respective completeness condition. For a short introduction, we refer the reader to \cite{Michor} or \cite{KMGroup}. Observe, that when we say  `tangent bundle', we  always refer to a kinematic tangent bundle, where the kinematic vector at a point is the velocity vector of smooth curves passing through this point. A kinematic vector field is a smooth section of the kinematic tangent bundle.
Observe that Frech\'et spaces are convenient vector spaces, and smoothness in this case coincides with $C^\infty$ smoothness with respect to the G\^ateaux derivative.
All smooth functions between manifolds $M$ and $B$ are denoted by $C^\infty(M,B)$, and if $B = \real$, we will simply write $C^\infty(M)$. All finite-dimensional manifolds will be Hausdorff and 2-nd countable and infinite-dimensional dimensional manifolds will be smoothly  Hausdorff.

Now, let us give two examples as a motivation for formulating and studying   sub-Riemannian infinite-dimensional manifolds.

\subsection{Riemannian submersions} \label{sec:RiemannianSubmersions}

Let $M$ and $B$ be possibly infinite-dimensional manifolds modeled on convenient vector spaces, and let $\pi\colon M \to B$ be a surjective map, such that the restriction of $d\pi$ to each tangent space is surjective. Such a map is called {\it submersion}. Assume that the kernel $\ker d\pi$ is a vector bundle and that there is another vector bundle $\calH$ on $M$, such that $TM$ is the Witney sum
$$TM = \ker d\pi \oplus \calH.$$
The sub-bundle $\calH$ is called  an {\it Ehresmann connection} of $\pi$. Furnish $B$ and $M$ with Riemannian metrics $\mathbf{b}$ and $\mathbf g$ respectively, such that $\ker d\pi$ and $\calH$ become orthogonal with respect to ${\bf g}$, and moreover,
\begin{equation} \label{eqRSHorizontal} 
\mathbf g(v_1, v_2) = \mathbf b( d_m\pi v_1, d_m\pi v_2), \qquad v_1, v_2 \in \calH_m.
\end{equation}
Then, the map $\pi\colon (M,\mathbf g) \to (B, \mathbf b)$ is called a {\it Riemannian submersion}. In this case, the Riemannian geodesics on $B$ are exactly the projections of the Riemannian geodesics on $M$, which are horizontal with respect to $\calH$ at one (and hence any) point~\cite{MM}. We use the term {\it Riemannian geodesic} for a curve $\gamma\colon [0,1]\to M$, which is a critical value of the energy functional $E(\gamma) = \frac{1}{2} \int_0^1 \mathbf g( \dot \gamma(t), \dot \gamma(t)) \, dt.$

Given a metric $\mathbf g$ on $M$, we can construct a Riemannian submersion in the following way. Define $\calH = (\ker d\pi)^\perp$ and assume that $\calH\oplus\ker d\pi=TM$. Then $\calH$ is an Ehresmann connection for the submersion $\pi\colon M\to B$. For any vector field $X$ on $B$, define $hX$ as a unique horizontal lift of $X$ to $M$, i.e.,  a unique vector field $hX$ with values in $\calH$ satisfying $d_m \pi (hX(m)) = X(\pi(m))$ for any $m \in M$. Then we can define the metric $\mathbf b$ on $B$ by
$$\mathbf b(X(b), Y(b)) = \mathbf g(hX(m), hY(m)), \qquad m \in \pi^{-1}(b).$$
Notice that a  submersion $\pi\colon M \to B$ with a metric $\mathbf g$ on $M$ can be considered as a Riemannian submersion, if and only if, $\mathbf g(hX(m), hY(m))$ does not depend on the choice of the element $m \in \pi^{-1}(b)$.

We can also construct a Riemannian submersion  starting with a Riemannian metric $\bf b$ on $B$. Choose a metric $\mathbf v$ on $\ker d\pi$ and a sub-bundle $\calH$ transversal  to $\ker d\pi$. Then the metric $\mathbf g$ can be defined by the relation
\begin{equation} \label{eqliftmetric} \mathbf g(v_1, v_2) = \mathbf b( d\pi v_1, d\pi v_2) + \mathbf v(\pr v_1, \pr v_2),\end{equation}
where $\pr\colon TM \to \ker d\pi$ is a projection satisfying $\ker \pr = \calH$.

If $B$ is a complex object and $M$ is a simpler one, then the Riemannian submersion $\pi\colon M \to B$ gives us a way to study the Riemannian geometry on a simpler object $M$ instead of $B$. Examples of results obtained using this technique in the study of the space of shapes can be found, e.g., in~\cite{Ellis, MM, MMVanish}.  

One can also define a metric only on $\calH$ and study the sub-Riemannian geometry on~$M$. Heuristically, it can be given by considering the metric space $(M, \mathbf{b})$ as a limiting case as $\varepsilon \to \infty$ of the punished metric $\mathbf g = \mathbf b \circ d\pi + \varepsilon (\mathbf v \circ \pr)$, where the expression is written in the sense of~\eqref{eqliftmetric}. In finite dimensions, this limit is realized in terms of Gromov-Hausdorff convergence of metric spaces, see e.g. \cite{Gromov}.

\subsection{Space of shapes and conformal welding} 

Let us consider a family of smooth two-dimensional shapes evolving in time. By {\it shape} we understand a simple closed smooth curve in the complex plane dividing it into two simply connected domains. The study of two-dimensional shapes is one of the central problems in the field of applied sciences. A program of such study and its importance was summarized by Mumford at ICM 2002 in Beijing \cite{Mumford}.
Let us consider a time-dependent family of shapes enclosing  bounded domains $\Omega(t)$ in $\comp$ representing a shape evolution  in the complex plane. Assume that all domains contain the origin  $0 \in \comp$. Typically, the study of the geometry of shapes resides in the study of analytic properties of a family of conformal embeddings $f(z,t)$ of the unit disk  $\unitD=\unitD_+$ into $\mathbb C$ such that $f(z,t)$ is a unique Riemann map of $\unitD$ onto $\Omega(t)$, that satisfies $f(0,t) = 0$, and $\partial_z f(0,t) >0$ for every $t\in [0, 1]$. We assume that $\partial\Omega(t)$ is $C^{\infty}$-smooth so  $f$ is smooth in $z$ up to $\partial \mathbb D_+$, and we assume also that $f(z,t)$ is smooth in  $t\in [0,1]$. Then for every such $f$ there is a matching function $g$ such that $g(z,t)$ maps the exterior $\unitD_{-}$ of the unit disk $\mathbb D_+$  onto the exterior of the domain $\Omega(t)$ and satisfies $g(\infty,t) = \infty.$ The superposition $f^{-1} ( g(e^{i\theta}, t ), t)$ is called a {\it conformal welding} for each fixed $t$.
We relate the motion of $f$ in time to the motion of $g$ by requiring
$$\frac{1}{2\pi} \int_0^{2\pi} \frac{\partial_t f^{-1} ( g(e^{i\theta}, t ), t)}{\partial_\theta f^{-1} ( g(e^{i\theta}, t ), t)} d\theta = 0.$$
See details in~\cite{GrGumVas}.
Here $f^{-1}(z,t)$ is the inverse function of $f(z,t)$ in $z$. We want to study the motion of $f(z,t)$ and $g(z,t)$ minimizing some energy that  depends only on the shape of the boundary $\partial \Omega(t)$. In Section \ref{sec:DiffVir} we will formulate this problem  as finding minimal horizontal curves with respect to a given distribution on the Virasoro-Bott group.

\section{Infinite-dimensional manifolds with constraints} \label{sec:MainSec}

\subsection{Sub-Riemannian geometry and geodesics in finite dimensions} \label{sec:sRfinite}

We will start by looking at the definition and basic properties of sub-Riemannian manifolds in finite dimensions.
Recall that a sub-Riemannian manifold, is a triple $(M, \calH, \mathbf h)$, such that $M$ is an $n$-dimensional connected smooth manifold, $\calH$ is a smooth sub-bundle of $TM$, and $\mathbf h$ is a Riemannian metric on $\calH$.
Often the smooth sub-bundle $\calH$ is considered as a smooth distribution which assigns to each point $m$ a linear subspace $\calH_m\subset T_mM$. We call 
$\calH$ a {\it horizontal distribution}. The pair $(\calH, \mathbf h)$ is called {\it a sub-Riemannian structure} on $M$.

 \begin{definition}\label{horcurve} 
An absolutely continuous curve $\gamma\colon I\to M$ is called $\calH$-horizontal, or simply horizontal if $\dot \gamma(t) \in \calH_{\gamma(t)}$ for almost all  $t\in I$.
 \end{definition}

For a pair of points $m_0, m_1\in M,$ let $\abscon_{\calH}(m_0,m_1)$ denote the collection of all horizontal absolutely continuous curves $\gamma\colon [0,1] \to M$ with square integrable derivatives that satisfy the boundary condition $\gamma(0) = m_0$ and $\gamma(1) = m_1$. Here, square integrability is defined with respect to the metric $\mathbf h$, however, any other choice of a metric on $\calH$ gives the same set of curves. Hence, the definition of $\abscon_{\calH}(m_0,m_1)$  depends only on $\calH$. The associated distance on $M$ corresponding to the sub-Riemannian structure $(\calH, \mathbf h)$ is given by
$$
d_{C-C}(m_0, m_1) = \inf \left\{ \int_0^1 \{{\mathbf h}(\dot \gamma(t), \dot \gamma(t))\}^{1/2} \, dt \, \colon \, \gamma \in \abscon_{\calH}(m_0,m_1) \right\}
$$
and is called the {\it Carnot-Carath\'eodoty distance}. 
The pair $(M, d_{C-C})$ forms a metric space, if and only if,  the distance $d_{C-C}$ is finite, or in other words, $\abscon_{\calH}(m_0,m_1)$ is non-empty for every pair of points $m_0,m_1 \in M$. Unlike usual Riemannian geometry, the map $m \mapsto d_{C-C}(m, m_1)$ is not smooth in general, and the Hausdorff dimension of the metric space $(M, d_{C-C})$ can be greater than the manifold topological dimension $n$.

The typical way to ensure that $\abscon_{\calH}(m_0,m_1)$ is nonempty for any pair $m_0,m_1\in M$, is to require that $\calH$ is {\it bracket generating}. To define this notion we denote by $\Gamma(\calH)$ sections of~$\calH$. Take $\calH^1 =\Gamma(\calH)$, and for any positive integer $k$, define
$$
\calH^{k+1} = \calH^{k} + [\calH, \calH^k].
$$
The collection of all obtained vector fields as $k\to\infty$ we denote by $\Lie \calH$.
Let $\Lie_{m} \calH$ be a subspace of $T_{m}M$ obtained by evaluating all the elements from $\Lie \calH$ at $m$. The distribution $\calH$ is called bracket generating if $\Lie_{m} \calH = T_{m}M$ for every $m \in M$. If $\calH$ is bracket generating, then  the Rashevski{\u\i}-Chow Theorem \cite{Chow, Rashevsky} guarantees that any two points can be connected by a horizontal curve. The metric topology induced by the Carnot-Carath\'eodory distance coincides with the manifold topology when $\calH$ is bracket generating. 

An important tool of defining the curves of minimal length is provided by  the Pontryagin Maximum Principle \cite{AS} that yields  the existence of two types of possible length minimizers, which are not mutually exclusive. 
The curves from the first type  minimizers are called {\it normal}. They are projections of solutions to a Hamiltonian system with a sub-Riemannian Hamiltonian function to the manifold. Locally, the sub-Riemannian Hamiltonian function is given by
\begin{equation} \label{eq:sRHamiltonian} 
H_{sR} (p)= \frac{1}{2} \sum_{j=1}^k h_{X_j}^2(p), \qquad h_{X_j}(p) := p(X_j(m)), \qquad p \in T^*_mM,
\end{equation}
where $(X_1, \dots, X_k)$ is a local orthonormal basis of vector fields from $\calH.$ A normal minimizer is always $C^\infty$-smooth and also  is local minimizer.

The other type of local minimizers consists of so-called {\it singular curves}, which can intuitively be thought of as  `bad points' of $\abscon_{\calH}(m_0,m_1)$. Namely, let $\abscon_{\calH}(m_0)$ be the collection of all horizontal absolutely continuous curves $\gamma\colon I \to M$, which are square integrable and satisfy only one-side boundary condition $\gamma(0) = m_0$. This is a Hilbert manifold modeled on $L^2(I, \real^k)$, where $k$ is the rank of $\calH$ \cite{Montsingular,Mon}. $\abscon_{\calH}(m_0, m_1)$ can then be identified with the preimage $(\End_{m_0})^{-1}(m_1)$ of the mapping
$$\begin{array}{rccc} \End_{m_0}: & \abscon_{\calH}(m_0) & \to & M \\
& \gamma & \mapsto & \gamma(1) \end{array}.$$
Hence, if $\gamma$ is a regular point of $\text{end}_{m_0}$, then   the space $\abscon_{\calH}(m_0,m_1)$ has the structure of a Hilbert manifold of codimension $n$ locally about $\gamma$  by the implicit function theorem.

\begin{definition}
An absolutely continuous horizontal curve $\gamma$ with $\gamma(0) = m_0$ is called singular, if it is a singular point of the mapping $\End_{m_0}$.
\end{definition}
The definition of singular curves  depends only on the sub-bundle $\calH$, and not on the metric $\mathbf h$. Singular curves are not always local minimizers, but all minimizers that are not normal, are singular curves. The term {\it abnormal} is also used for singular curves. It is still an open question whether all singular curves, which are minimizers at the same time, are smooth. Some results in this direction, and on singular curves in general, can be found, e.g.,~in~\cite{AgrachevSarychev,BonnardT,BH,ChtourJT,GoleK,Montsingular,Mon}.

\begin{remark} 
Some authors prefer to use Lipschitz curves instead of square integrable curves. The collection of curves starting at a fixed point $m_0$ then becomes a Banach manifold modeled on $L^\infty(I, \real^k).$ 
\end{remark}

\subsection{Sub-Riemannian  infinite-dimensional manifolds}
In order to generalize the definition of a finite-dimensional sub-Riemannian manifold to infinite dimensions, we need  an extra requirement.

\begin{definition}\label{def:sR}
A sub-Riemannian manifold is a triple $(M, \calH, \mathbf h)$, where
\begin{itemize}
\item $M$ is a connected manifold modeled on $c^{\infty}$-open sets of a convienient vector space;
\item $\calH$ is a splitting sub-bundle of $TM$, i.e.,  there exists another sub-bundle $\calV$, such that
\begin{equation} \label{eq:HorVert} 
TM = \calH \oplus \calV; 
\end{equation}
\item $\mathbf h$ is a weak metric on $\calH$.
\end{itemize}
\end{definition}
Here `week' means that the mapping $v \in \calH_m \mapsto \mathbf h(v, \cdot) \in \calH_m^*$ is injective but not necessarily surjective.
The requirement of the splitting condition is non-trivial if $M$ is not modeled on a Hilbert space, see~\cite{LT}. In particular, it implies that there exists a smooth projection from $TM$ to $\calH$. All extra requirements in Definition~\ref{def:sR} are always satisfied in the finite-dimensional case.

We restrict ourselves to considering only smooth curves not only in a way of simplification, but also because enlarging the space of curves does not guarantee a nicer topology of this space.   Hence, we will use the term {\it horizontal curve}  meaning a smooth curve $\gamma\colon I \to M$, such that $\dot \gamma \in \calH_{\gamma(t)}$ for every $t\in I$. Denote the collection of all such curves by $C^\infty_{\calH}(I, M)$.

Now, let us make use of calculus of variations. We say that a smooth map $\Phi\colon I \times (-\epsilon, \epsilon) \to M$ is a {\it variation} of a curve $\gamma\in C^\infty(I,M)$ if 
\begin{equation}\label{def:genvar}
\Phi(t,0) = \gamma(t), \quad \Phi(0,s) = \gamma(0), \quad \text{ and } \quad \Phi(1,s) = \gamma(1).
\end{equation} 
For a fixed $s$, let us denote by $\gamma^s$ the curve $t \mapsto \Phi(t,s)$. The map $s \mapsto \gamma^s$ can be seen as a curve in $C^\infty(I,M)$. By slight abuse of notations, we will denote the variation simply by $\gamma^s$. We say that a variation is $\calH${\it-horizontal},  if for each $s\in (-\epsilon, \epsilon)$, the curve $\gamma^s(t)$, $t\in I$, is $\calH$-horizontal. Denote by  $\calJ_{\calH}(\gamma)$  the collection of all $\calH$-horizontal variations of $\gamma$.

Observe that the problem of length minimization is equivalent to the problem of energy minimization, which allows us to formulate the first-order condition for a length minimizer as follows.

\begin{definition}
Let us define the sub-Riemannian energy functional on $C^\infty_{\calH}(I,M)$ by $E(\gamma) = \frac{1}{2} \int_0^1\mathbf h(\dot \gamma, \dot \gamma) \, dt.$ An $\calH$-horizontal curve $\gamma$ is  called a sub-Riemannian geodesic if
$$\partial_s E(\gamma^s)\big|_{s=0} = 0, \quad \text{ for any }\quad \gamma^s \in \calJ_{\calH}(\gamma).$$
\end{definition}

It is difficult to  compute such curves explicitly  in a most general setting without additional assumptions even in the Riemannian case $\calH = TM$. Therefore, we want to study some particular cases where the solutions exist in the Riemannian case, and see then, whether it helps to find formulas for the sub-Riemannain geodesics. This usually means that we must choose a way to extend the metric $\mathbf h$ to the entire tangent bundle.

\begin{definition}
Let $(M,\calH, \mathbf h)$ be a sub-Riemannian manifold. A Riemannian metric $\mathbf g$ on $M$ is said to tame $\mathbf h$ if $\mathbf g|_{\calH} = \mathbf h$, and the orthogonal complement $\calH^\perp$ to $\calH$ with respect to $\mathbf g$ is a sub-bundle satisfying $\calH \oplus \calH^\perp= TM$.
\end{definition}

Let $\calV$ be a vector bundle such that $\calH \oplus \calV = TM$. Assume that there exists a metric $\mathbf v$ on $\calV$, and define the metric $\mathbf g = \mathbf h \oplus \mathbf v$, i.e.,  $\calH$ and $\calV$ become orthogonal with respect to $\mathbf g$ and $\mathbf g|_{\calH} = \mathbf h$,  $\mathbf g |_{\calV} = \mathbf v$. We conclude that a Riemannian metric $\mathbf g$ which tames a sub-Riemannian metric $\mathbf h$ exists if and only if the horizontal sub-bundle $\calH$ has a complement sub-bundle that admits a metric.

\begin{remark}
In contrast to finite-dimensional Riemannian geometry, a distance given by a (weak) Riemannian metric may vanish between some distinct points. See~\cite[section 3.10]{MM},~\cite{MMVanish}, for examples. This implies that if we define the Carnot-Carath\'eodory distance by
$$d_{C-C}(m_0,m_1) = \inf\left\{ \int_0^1 \{\mathbf h(\dot \gamma, \dot \gamma)\}^{1/2} \,dt \, : \, \gamma\in C^{\infty}_{\calH}([0,1],M),\  \gamma(0) =m_0,\ \gamma(1) = m_1\right\},
$$
it is possible that it may vanish for some points as well being a generalization of the Riemannian distance.
\end{remark}

\begin{remark}
Although the variational approach has been used in sub-Riemannian geometry in finite dimensions, see, e.g., \cite{Hamenstadt}, usually the Hamiltonian viewpoint  is preferred, as it does not require a choice of the Riemannian metric to tame $\mathbf h$. The reason is that whereas there is no canonical choice of the complement to $\calH$ in $TM$, the sub-bundle $\Ann(\calH) = \{p \in T^*_mM \, : \,  p(v) = 0 \text{ for any } v \in \calH_m\ m \in M\}$ is canonical. Having only a weak metric, we try to avoid  cotangent bundles because we can not associate elements in $T^*_mM$ to $T_mM$ any longer by using a metric.
\end{remark}

We are aimed at computing sub-Riemannian geodesics with respect  to a metric $\mathbf h$, provided a sufficiently nice Riemannian metric $\mathbf g$ that tames $\mathbf h$. However, we need a new definition to describe horizontal curves, which can be geodesics but which do not depend on the metric $\mathbf h$, and depend only on the horizontal sub-bundle itself. They are, in some sense, counterparts of singular curves in finite dimensions.

\subsection{Semi-ridig curves for  infinite-dimensional sub-Riemannian manifolds} \label{sec:SemiRigid}

The definition for singular curves can not be extended to general  infinite-dimensional  manifolds modeled on convenient vector spaces. Therefore, we propose a way to determine curves, which depend only on the distribution $\calH$. Let $\Vect(\gamma) := \Gamma(\gamma^*(TM))$ denote the space of smooth vector fields along $\gamma$. Put $m_0 = \gamma(0)$, $m_1 = \gamma(1)$, and use $C^\infty_{\calH}(I,M;m_0,m_1)$ for the subset of $C^\infty_{\calH}(I,M)$ containing curves starting at $m_0$ and ending at $m_1$. Then, although there could be no manifold structure on $C^\infty_{\calH}(I,M;m_0,m_1)$, heuristically, we may think of the collection of curves $\gamma^s\in \calJ_{\calH}(\gamma)$ having the same derivative $\partial_s \gamma^s |_{s=0}= Z(t) \in \Vect(\gamma)$ as an equivalence class of curves in  $C^\infty_{\calH}(I,M;m_0,m_1)$ representing a tangent vector at $\gamma$. Then `bad curves', which we will call semi-rigid, can be considered as curves where `the tangent space is too small'. Let us provide the rigorous meaning to the above sentence.

If a curve $s \to \gamma^s$ is a variation of $\gamma^0 = \gamma$, i.e., it fixes the endpoints of $\gamma$, then it is clear that any vector field along $\gamma$ obtained by $Z(t) := \partial_s\gamma^s|_{s=0}$, must vanish at the endpoints. If in addition, a variation  is horizontal, then we want to find an additional condition for $Z$ related to the curve horizontality property. Recall that the {\it canonical flip} $\jmath$ is a unique vector bundle isomorphism making the following diagram commute
$$\xymatrix{T(TM) \ar[rr]^\jmath \ar[rd]_{\pr_{TM}} & & T(TM) \ar[ld]^{d(\pr_{M})} \\ & TM.}$$

First let us observe that here, $\calH$ and consequently $T\calH$, are viewed as sub-manifolds of $TM$ and $T(TM)$ respectively. We remark that although $T\calH$ considered  this way, will not be a sub-bundle of the vector bundle $T(TM)$, its image under the canonical flip will have this property, and hence, the concept of $\jmath(T\calH)$-horizontality on the manifold $TM$ is well defined.
Now we are ready to formulate the following statement.

\begin{lemma}
Let $s \mapsto \gamma^s$ be a smooth curve in $C^\infty(I,M)$, defined in an interval $(-\varepsilon, \varepsilon)$ with $\gamma^0 = \gamma$. Assume that for each fixed $s$, $\gamma$ is $\calH$-horizontal, and define
$$Z(t):= \partial_s \gamma^s(t) |_{s=0}\in \Vect(\gamma).$$
Then the curve $t \mapsto Z(t)$ in $TM$ is $\jmath(T\calH)$-horizontal.
\end{lemma}

\begin{proof}
Since $\gamma^s$ is horizontal for any $s$, we know that $\dot\gamma^s(t) \in \calH_{\gamma^s(t)}$ for any $s,t$. In addition, the derivative of the curve $Z(t)$ in $TM$ satisfies
$$\partial_t Z(t) = \partial_t \partial_s \gamma^s(t)|_{s=0} = \jmath(\partial_s \partial_t \gamma^s(t))|_{s=0} = \jmath(\partial_s \dot\gamma^s(t)|_{s=0}).$$
Clearly, $\partial_s \dot\gamma^s(t)|_{s=0}$ is a tangent vector to $\calH$ at the point $\dot \gamma(t)$. Hence, $\partial_t Z(t) \in \jmath(T\calH)$ for any $t$.
\end{proof}

We will denote the space of all vector fields $Z$ along $\gamma$ that are $\jmath(T\calH)$-horizontal by $\Vect_{\calH}(\gamma).$ Furthermore, let us write $\Vect^{fix}_{\calH}(\gamma)$ for the subspace of $\Vect_{\calH}(\gamma)$ consisting of vector fields satisfying
$$X(0) = \vec{0}_{\gamma(0)}, \qquad X(1) = \vec{0}_{\gamma(1)}.$$
The expression $\vec{0}_m$ denotes the zero element in $T_mM$. Finally, we define 
\begin{equation}\label{def:varH}
\Var_{\calH}(\gamma)=\left\{ Z \in \Vect(\gamma) \, \colon \, Z(t) = \partial_s \gamma^s(t) |_{s=0} \text{ for some } \gamma^s \in \calJ_{\calH}(\gamma) \right\}.
\end{equation}

The sets $\Vect^{fix}_{\calH}(\gamma)$ and $\Var_{\calH}(\gamma)$ are not the same in general, i.e.,  not all  vector fields $X \in \Vect^{fix}_{\calH}(\gamma)$ can by obtained from some horizontal variation. This fact brings us to the following definition.

\begin{definition}
We say that a curve $\gamma\in C^{\infty}_{\calH}(I,M)$ is semi-rigid, if $\Var_{\calH}(\gamma)$ is a proper subset of $\Vect_{\calH}^{fix}(\gamma)$.
\end{definition}

Notice that the definition of semi-rigid curves depend on $\calH$ only, and does not invoke the metric on $\calH$ in any way. 
The following example in finite dimensions  justifies the term semi-rigid. 

\begin{example}
Let $\real^3$ be the Euclidean space with coordinates $(x, y, z)$, and let $\calH$ be a distribution spanned by the vector fields
$$
X = \partial_{x} - \tfrac{1}{2} y^2 \partial_{z} \quad \text{ and } \quad Y = \partial_{y}.
$$
The distribution $\calH$ is called the Martinet distribution. A curve $\gamma(t) = (x(t), y(t), z(t))$ is horizontal if
\begin{equation} \label{eq:Martinet} 
\dot z = - \tfrac{1}{2} \,  y^2 \, \dot x.
\end{equation}
Let $Z \in \Vect(\gamma)$ be written in coordinates as $Z(t) = u(t) \partial_{x}|_{\gamma(t)}  + v(t) \partial_{y}|_{\gamma(t)} + w(t) \partial_{z}|_{\gamma(t)}$. Considering an $\calH$-horizontal variation of $\gamma$, we deduce that $t \mapsto Z(t)$ is in $\Vect_{\calH}(\gamma)$ if, in addition to~\eqref{eq:Martinet}, we have

\begin{equation} \label{eq:MartinetTH} 
\dot w =  -\tfrac{1}{2} y^2 \dot u - y v \dot x.
\end{equation}

Take a particular choice of $\hat\gamma(t) = (x(t), y(t), z(t)) = (t,0,0)$, $t\in[0,1]$.  It is known that this curve is singular,  see, e.g.,~\cite[Section~3.3]{Mon}. Moreover,  it is a local length minimizer with respect to any metric $\mathbf h$ on $\calH$. However, it is not a normal minimizer for a generic choice of $\mathbf h$. Let us show that it is semi-rigid. Pick any vector field $Z\in\Vect_{\calH}^{fix}(\hat\gamma)$.  Then, by~\eqref{eq:MartinetTH}, we obtain that $\dot w(t)=0$. The condition $Z(0) = \vec{0}_{\hat\gamma_0}$ implies that $w(t)=0$ for all $t$. Thus, the general form of $Z(t)$ is  
\begin{equation}\label{Z}
Z(t)= u(t) \partial_u|_{\hat\gamma(t)} + v(t) \partial_{v} |_{\hat\gamma(t)}\quad\text{with}\quad u(0) = v(0) = u(1) = v(1) = 0.
\end{equation}
Now let us show that there is no variations $\gamma^s\in \calJ_{\calH}(\hat\gamma)$ except of a reparametrizaition of~$\hat\gamma$. Choose any $\gamma^s = (x^s, y^s, z^s)$ from $\calJ_{\calH}(\hat\gamma)$. If a vector field $Z\in\Vect_{\calH}^{fix}(\hat\gamma)$ were obtained from the variation $\gamma^s$, then we would have
$$
x^s(t) = t+ su(t) + o(s), \qquad y^s(t) = sv(t) + o(s), \qquad z^s(t) = o(s),
$$
for some functions $u,v$ satisfying $u(0) = v(0) = u(1) = v(1) = 0$. However, integrating~\eqref{eq:Martinet}, we obtain the formula
$$z^s(1) = -\frac{1}{2} s^2 \int_0^1 v(t)^2 \, dt + o(s^2).$$
The value of $z^s(1)$ is strictly negative for a sufficiently small $s$, unless $v \equiv 0$. So we conclude that if $\gamma^s\in \calJ_{\calH}(\hat\gamma)$, then $Z(t) = \partial_s \gamma^s(t) |_{s=0}$ can  hold only if
$Z(t) = u(t) \partial_x |_{\hat\gamma(t)}$.
\end{example}

As an additional information,  the above example shows that any $\gamma^s\in \calJ_{\calH}(\hat\gamma)$ for $\hat\gamma(t) = (t , 0, 0)$ is a reparametrization of $\hat\gamma$. Such kind of curves in literature are called {\it rigid} or $C^1${\it{-rigid}}. Intuitively this means that a rigid curve can not be deformed by any means  keeping endpoints fixed without loosing $\calH$-horizontality. This is our motivation for the terminology semi-rigid. A semi-rigid curve, in general, can be deformed but possibly not in all directions. Obviously, rigid curves are semi-rigid except for the trivial case when the horizontal sub-bundle $\calH$ is of rank $1$.

The results of~\cite[p. 439]{BH} show that  if  the sets $\Var_{\calH}(\gamma)$ and $\Vect^{fix}_{\calH}(\gamma)$  coincide for a curve $\gamma$, then the curve is regular, or in our terminology, is not semi-rigid.
Reversing this statement we come to the following conclusion.
\begin{proposition}
Semi-rigid curves are singular.
\end{proposition}

It is worth noticing that the gap in the inclusion $\Var_{\calH}(\gamma) \subseteq \Vect^{fix}(\gamma)$ was observed  before (e.g., \cite{Hamenstadt, Mon94, Montsingular, Mon}), which essentially led to the study of singular, abnormal, and especially, rigid curves. The endpoint map and the Pontryagin Maximum Principle are the crucial tools, which are not available in the case of infinite-dimensional manifolds, therefore, we  give definitions using only the presence or absence of variational vector fields. 


\subsection{Local viewpoint  through adjoints}\label{sec:InfiniteMF}

Let $M$ be a manifold,  $(\calH, \mathbf h)$ be a sub-Riemannain structure on $M$, and let $\mathbf g$ be a Riemannian metric taming $\mathbf h$. We denote $\calV = \calH^\perp$ and choose a bundle chart in a neighborhood $U\subset M$: 
$$\begin{array}{ccc} TU & \to & U \times V \\
v \in T_mU & \mapsto & (m, \theta(v)) \end{array},$$
where  $V$ is some convenient vector space,
such that 
\begin{itemize}
\item[1)]{there is a splitting $V = \calH_0 \oplus \calV_0$ satisfying
$$\theta^{-1}(\calH_0) = \calH \cap TU, \qquad \theta^{-1}(\calV_0) = \calV \cap TU;$$}
\item[2)]{there exists an inner product $\langle \cdot , \cdot \rangle$ on $V$, satisfying
$$\mathbf g(v_1, v_2) = \langle \theta(v_1), \theta(v_2) \rangle.$$}
\end{itemize}
We can always assume 1), but this is not necessarily true for the second assertion.
If there is a basis of orthogonal vector fields  in $U$, then we can use this basis to construct a bundle chart satisfying 2). We consider $\theta$ as an $V$-valued one-form on $U$, and $d\theta$ to be the exterior differential of $\theta$.

Further, we make the following assumptions on $\theta$ and $\langle \cdot, \cdot \rangle$:
\begin{itemize}
\item[(A)] There is a bilinear map $a^\top: V \times V \to V$, satisfying
$$\langle d\theta(v_1, v_2), u \rangle = \langle \theta(v_2), a^\top(\theta(v_1), u) \rangle, \qquad v_1, v_2 \in T_mM, \quad u \in V.$$
The notation $a^{\top}$ is introduced by similarity with the adjoint to $d\theta$;
\item[(B)] For a chosen curve $\gamma\in C^{\infty}(I,M)$, we define a map $\Xi_{\gamma}: \Vect(\gamma) \to C^\infty(I,V),$ by
$$\Xi_{\gamma}(X)(t) = \partial_t \theta(X(t)) - d\theta(\dot \gamma, X(t)),\ \ t\in I.$$
We suppose that for any $y \in C^\infty(I,V)$, the Cauchy problem
$$\Xi_\gamma(X) = y, \qquad X(0) = \vec{0}_m,$$
has a unique solution $X=\Xi_{\gamma}^{-1} y $.
\end{itemize}

Given these assumptions, we look for sub-Riemannian geodesics among the curves of two types: semi-rigid curves for which $\Var_{\calH}(\gamma)$ is a proper subset of $\Vect_{\calH}^{fix}(\gamma)$, and the other ones for which $\Var_{\calH}(\gamma)=\Vect_{\calH}^{fix}(\gamma)$. The main result is the following.

\begin{theorem} \label{theorem:main}
Assume that $\gamma$ is a sub-Riemannian geodesic on $(M,\calH, \mathbf h)$. Then either $\gamma$ is semi-rigid or there is a curve $\lambda \in C^\infty(I, \calV_0)$, such that $\lambda$ and $\gamma$ satisfy the system of equations
\begin{equation} \label{eq:Normalgeod} \theta(\dot \gamma) = u, \qquad \dot u = -\pr_{\calH_0} a^\top(u,u+ \lambda), \qquad \dot \lambda = -\pr_{\calV_0} a^\top(u, u+ \lambda).\end{equation}

\smallskip
Conversely, any curve $\gamma\in C^{\infty}_{\calH}(I,M)$ satisfying system~\eqref{eq:Normalgeod} is a sub-Riemannian geodesic. A semi-rigid curve does not need to be a geodesic.
\end{theorem}

We emphasize that the `or' in Theorem~\ref{theorem:main} is not exclusive. A sub-Riemannian geodesic may be  semi-rigid and may satisfy~\eqref{eq:Normalgeod} at the same time. Inspired by this theorem, we give the following definition of normal geodesics and show in Section~\ref{sec:NormalFD} that for all finite-dimensional Riemannian manifolds our new definition coincides with the classical one.

\begin{definition}\label{def:normalsRg}
A sub-Riemannian geodesics $\gamma$, which is a solution to~\eqref{eq:Normalgeod} for some $\lambda \in C^\infty(I, \calV_0)$ is called normal.
\end{definition}

\begin{proof}[Proof of Theorem~\ref{theorem:main}]
We start from two general observations and then apply them to our particular situation.

{\sc Observation I.}
Define an inner product in the space $C^\infty(I,V)$  by
$$\llangle x, y \rrangle = \int_0^1 \langle x(t), y(t) \rangle \, dt.$$

We extend the definition of energy to all curves by the formula $E(\gamma) = \frac{1}{2} \int_0^1 \mathbf g(\dot \gamma(t), \dot \gamma(t)) \, dt.$ Let $\gamma$ be an arbitrary, not necessarily horizontal, curve $\gamma\in C^{\infty}(I,M)$, and let $\gamma^s$ be its variation in the sense of~\eqref{def:genvar}. Define 
$u^s(t) = \theta(\dot \gamma^s(t))$ and $Z(t)=\partial_s \gamma^s(t) |_{s=0}$.
If we denote by $[\gamma^s(t)]^*$  the pullback along the map $(t,s) \mapsto \gamma^s(t)$, then
\begin{align} \label{eq:uXiZ}
\partial_s u^s(t) |_{s=0}  & = \partial_s \theta(\partial_t \gamma^s(t))|_{s=0} = \partial_s [\gamma^s(t)]^*\theta(\partial_t) |_{s=0}\\ \nonumber
& = \left. \Big(\partial_t [\gamma^s(t)]^*\theta(\partial_s) - d [\gamma^s(t)]^*\theta(\partial_t, \partial_s) \Big) \right|_{s=0} \\ \nonumber
& = \partial_t \theta(Z(t)) - d \theta(\dot \gamma(t), Z(t)) = \Xi_{\gamma}(Z)(t).
\end{align}
Therefore, writing $u = \theta(\dot \gamma)$, we obtain
\begin{equation}\label{eq:orthoguZ}
\partial_s E(\gamma^s) |_{s=0} = \int_0^1 \langle u(t), \partial_s u^s(t) |_{s=0} \rangle \, dt = \int_0^1 \langle u(t), \Xi_{\gamma}(Z)(t) \rangle \, dt = \llangle u, \Xi_{\gamma}(Z) \rrangle
\end{equation}
for any vector field $Z$ associated with the variation $\gamma^s$.

{\sc Observation II.}
Define 
$$
\Vect^{fix}(\gamma) = \left\{ X \in \Vect(\gamma) \, \colon \, X(0) = \vec{0}_{\gamma(0)}, 
 X(1) = \vec{0}_{\gamma(1)}\right\},
 $$
and let $X \in \Vect^{fix}(\gamma)$ and $y\in \Big(\Xi_{\gamma} \Vect^{fix}(\gamma) \Big)^\perp$, where the orthogonal complement is taken with respect to the product $\llangle \cdot, \cdot \rrangle$. Then the following equality 
\begin{align*}
0 = \llangle y, \Xi_{\gamma}(X) \rrangle & = \int_0^1 \langle y(t), \partial_t\theta(X(t)) - d\theta(\dot \gamma(t), X(t)) \rangle dt \\
& = - \int_0^1 \Big\langle \dot y(t) + a^\top(\theta(\dot \gamma(t)), y(t)), \theta(X(t)) \Big\rangle \, dt= - \llangle \dot y + a^\top(u,y),  \theta(X) \rrangle
\end{align*}
holds. Since $X$ is chosen arbitrarily, the curve $y$ is a solution to $\dot y = - a^\top(u,y)$. 

Now let $\gamma$ be a sub-Riemannian geodesic and $\gamma^s\in\calJ_{\calH}(\gamma)$. Then $Z=\partial_s \gamma^s(t) |_{s=0}\in\Var_{\calH}(\gamma)$ by~\eqref{def:varH}. Moreover, $u$ and $\partial_s u^s |_{s=0}\in C^\infty(I, \calH_0)$, and relation~\eqref{eq:uXiZ} implies that $\Xi_{\gamma}(Z)$ is also from $C^\infty(I, \calH_0)$.
We conclude by~\eqref{eq:orthoguZ} that $\gamma$ is a sub-Riemannian geodesic, if and only if, $u\in\Big(\Xi_{\gamma} \Var_{\calH}(\gamma)\Big)^{\bot}$ in $C^\infty(I, \calH_0)$ with respect to the inner product $ \llangle \cdot , \cdot \rrangle$.

The inclusion $\Var_{\calH}(\gamma)\subseteq \Vect^{fix}_{\calH}(\gamma)$ implies $\Big(\Xi_{\gamma}\Var_{\calH}(\gamma)\Big)^{\bot}\supseteq\Big(\Xi_{\gamma}\Vect^{fix}_{\calH}(\gamma)\Big)^{\bot}$. We consider two cases 
\begin{itemize}
\item[a)]{$u\in \Big(\Xi_{\gamma}\Vect^{fix}_{\calH}(\gamma)\Big)^{\bot}$,}
\item[b)]{$u\in\Big(\Xi_{\gamma}\Var_{\calH}(\gamma)\Big)^{\bot}$ but not in $\Big(\Xi_{\gamma}\Vect^{fix}_{\calH}(\gamma)\Big)^{\bot}$.}
\end{itemize}

Case a). Observe that $\Vect^{fix}_{\calH}(\gamma) = \Xi^{-1}_{\gamma} \pr_{\calH_0} \Xi_{\gamma}\Vect^{fix}(\gamma)$. Hence, we obtain that 
\[
u\in\Big(\pr_{\calH_0} \Xi_{\gamma} \Vect^{fix}(\gamma) \Big)^\perp=\pr_{\calH_0} \Big(\Xi_{\gamma} \Vect^{fix}(\gamma) \Big)^\perp,
\]
where the orthogonal complement is taken with respect to $\llangle \cdot , \cdot \rrangle$, but the first one in the space $C^{\infty}(I,\calH_0)$ and the second one in $C^{\infty}(I,V)$.

Let $y$ be an arbitrary element in $\Big(\Xi_{\gamma} \Vect^{fix}(\gamma) \Big)^\perp$. Then  for any $X \in \Vect^{fix}(\gamma)$ Observation II implies that $y$ is a solution to $\dot y = - a^\top(u,y)$. Now set $u=\pr_{\calH_0} y$ and $ \lambda=\pr_{\calV_0} y$ in order to obtain~\eqref{eq:Normalgeod}. We conclude that in this case the sub-Riemannian geodesic $\gamma$ is normal. 

Case b). There is $Y \in \Vect^{fix}_{\calH}(\gamma)$ such that $\llangle u, \Xi_{\gamma}(Y) \rrangle \neq 0$, but $\llangle u, \Xi_{\gamma}(X) \rrangle = 0$ for any $X \in \Var_{\calH}(\gamma)$. So the inclusion $\Var_{\calH}(\gamma) \subseteq \Vect^{fix}_{\calH}(\gamma)$ is proper, because $Y$ cannot be in $\Var_{\calH}(\gamma)$. Thus the sub-Riemannian geodesic $\gamma$ is semi-rigid curve in this case. 

To show the converse statement to Theorem~\ref{theorem:main}, we choose an arbitrary curve $\gamma\in C^{\infty}_{\calH}(I,M)$, that satisfies system~\eqref{eq:Normalgeod} for some $\lambda\in C^{\infty}(I,\calV_0)$. Then $y=u+\lambda$ satisfies the equation $\dot y = - a^\top(u,y)$ by linearity of $a^{\top}$. Observation II yields that $\gamma$ is a sub-Riemannian geodesic. 
\end{proof}

\subsection{Comparison with the finite-dimensional case} \label{sec:NormalFD} 

Although the assumptions made in Section~\ref{sec:InfiniteMF} might seem very specific, we will show that all finite-dimensional sub-Riemannian manifolds can locally be described this way. We also show that Definition~\ref{def:normalsRg} of normal geodesics coincides with the one given in Section~\ref{sec:sRfinite}, justifying the terminology. Observe, that in finite dimensions, the normal geodesics are local minimizers, hence we loose nothing by restricting the considerations to an arbitrarily small neighbourhood. 

Let $(M, \calH, \mathbf h)$ be an arbitrary $n$-dimentional sub-Riemannian manifold, where $\calH$ has rank $k$, and let $\mathbf g$ be a metric taming $\mathbf h$. Let $U$ be a sufficiently small neighborhood, such that there exists an orthonormal with respect to $\mathbf g$ basis $X_1, \dots, X_n$ of vector fields on~$U$. From this basis choose vector fields $X_1, \dots, X_k$, such that they span~$\calH|_U$, and pick up a corresponding co-frame $\theta_1, \dots, \theta_n$. Then the form $\theta = (\theta_1, \dots, \theta_n)$ is $\real^n$- valued one-form. We extend the class of smooth curves by including absolutely continuous curves, and remark that a curve $\gamma$ is horizontal if and only if $\theta(\dot \gamma(t))$ is contained in $\real^k \times \{0\}$ for almost all $t$.

Let $\abscon(m_0)$ be the collection of all absolutely continuous square integrable curves in $U$ starting at $m_0$. Then the map
$$
\begin{array}{ccccc}
\Theta\colon & \abscon(m_0) &\to & L^2(I,\real^n)
\\
& \gamma &\mapsto &\theta(\dot \gamma)
\end{array}
$$
is a diffeomorphism onto a neighborhood of $0 \in L^2(I, \real^n)$, see \cite[Lemma 2.1]{Hamenstadt}. It can be easily verified that $d_{\gamma}\Theta = \Xi_{\gamma}$ in this case, and the mapping $\Xi_{\gamma}$ is invertible, since $\Theta$ is a diffeomorphism.

In coordinates, it admits the following form. If $x = (x_1, \dots, x_n)\in \real^n$, then for any pair for vector fields $Y$ and $Z$ on $U$ with $\theta(Y) = y = (y_1, \dots, y_n)$ and $\theta(Z) = z = (z_1, \dots, z_n)$, we have
$$
\langle d\theta(Y, Z ) , x \rangle = \sum_{i,j,l=1}^n x_i y_j z_l(\Gamma_{lj}^i - \Gamma_{jl}^i), \qquad \Gamma_{jl}^i := \mathbf g(\nabla_{X_j} X_l, X_i).
$$ 
Here $\langle \cdot,\cdot \rangle $ is the standard Euclidean inner product. We conclude that the adjoint map is given by
$$a^{\top}(y)x = \sum_{j,l=1}^n x_l y_j \Big(\Gamma_{1j}^l- \Gamma_{j1}^l, \Gamma_{2j}^l- \Gamma_{j2}^l, \dots, \Gamma_{nj}^l- \Gamma_{jn}^l \Big).$$

The following proposition justifies the use of the term `normal sub-Riemannian geodesic'. Let a sub-Riemannian Hamiltonian function $H_{sR}(m,p)$ be given by~\eqref{eq:sRHamiltonian}.
 As it  was mentioned before, all such curves are always smooth local minimizer with respect to the metric $d_{C-C}$~\cite{Mon}.
We reserve first $k$ coordinates in $\real^n$ for the image of $\calH$ under $\theta$.

\begin{proposition}
A horizontal curve $\gamma\colon I\to U$ is a projection of a solution to the Hamiltonian system associated with the Hamiltonian function~\eqref{eq:sRHamiltonian}, if and only if, $\gamma$ is a solution to system~\eqref{eq:Normalgeod} for some curve $\lambda\colon I\to {0} \times \real^{n-k}$. 
\end{proposition}
\begin{proof}
Let us introduce the coordinates on the cotangent bundle $T^*M$ by writing
$p = \sum_{j=1}^n p_j \theta_j(m)$ for any $p \in T^*_mM$. Notice, that since $p_j = h_{X_j}(p)=p(X_j(m))$ by~\eqref{eq:sRHamiltonian}, the sub-Riemannian Hamiltonian can be written as $H_{sR}(p) = \sum_{j=1}^n p_j^2.$ As a consequence we arrive at
$$\left( \frac{\partial H_{sR}}{\partial p_1}, \dots, \frac{\partial H_{sR}}{\partial p_n} \right) = (p_1, \dots, p_k, 0, \dots, 0),$$
\begin{align*} \{h_{X_i}, H_{sR}\}(p)  &= -\sum_{j=1}^k h_{X_j}(p) h_{[X_i, X_j]}(p) = -\sum_{j,l=1}^k p_j p_l (\Gamma_{ij}^l - \Gamma_{ji}^l).\end{align*}
Let $t \mapsto p(t)$ be a curve in $T^*U$ that is projected to $\gamma$ with $p_j(t) = h_{X_j}(p(t))$, and let us write $u(t) = (p_1(t), \dots, p_k(t), 0 \dots, 0)$ and $\lambda(t) = (0, \dots, 0, u_{k+1}(t), \dots, u_n(t))$. Then $t \mapsto (\gamma(t), p(t))$ is a solution to the Hamiltonian system, i.e., it satisfies
$$\theta_i(\dot \gamma) = \frac{\partial H_{sR}}{\partial p_i}, \qquad \dot p_j = \{h_{X_i}, H_{sR}\}(p),$$
if and only if, $u = \theta(\dot \gamma)$, and
$$\dot u +Ê\dot \lambda =- \sum_{j,l=1}^k p_j p_l \Big(\Gamma_{1j}^l - \Gamma_{j1}^l, \dots, \Gamma_{nj}^l - \Gamma_{jn}^l\Big) = - a^\top(u, u + \lambda).$$
\end{proof}

\begin{corollary}
Definition~\ref{def:normalsRg} and the definition of normal geodesics given in Section~\ref{sec:sRfinite} coincide.
\end{corollary}

In Section~\ref{sec:GroupCritical} we also show that all infinite-dimensional regular Lie groups with an invariant (either left or right) sub-Riemannian structure  also possess conditions (A) and (B). 

\subsection{Connectivity by horizontal curves} \label{sec:controllability}
Apart from the optimality conditions for horizontal curves, we also need to discuss a possibility to connect two arbitrary points by a smooth $\calH$-horizontal curve. This problem is often called {\it controllability} in the theory of geometric control. 

As we mentioned before, the Rashevski{\u\i}-Chow Theorem~\cite{Chow, Rashevsky} in finite dimensions states that if $\calH$ is bracket generating and a manifold $M$ is connected, then any pair of points can be connected by an absolutely continues horizontal curve. The statement remains true if we additionally require horizontal curves to be smooth. It is still an open question whether the bracket generating condition implies the existence of a smoothly immersed curve connecting two points. A generalization of the Rashevski{\u\i}-Chow theorem to infinite-dimensional manifolds is a challenging problem. The only result in this direction we are aware of~\cite{Ledyaev} asserts that if a horizontal distribution is bracket generating on a Hilbert manifold $M$, then the set reachable by horizontal curves with a fixed starting point is dense in $M$. See also~\cite{Dubnikov,Heintze} for some progress in Hilbert and Banach manifolds.

There is also a  statement where bracket generating plays a role in the problem of controllability for certain choices of horizontal distributions on diffeomorphism groups. Let $M$ be an $n$-dimensional compact manifold. Then the group $G = \Diff M$ of diffeomorphisms of $M$ is a Lie-Fr\'echet group under the group operation of superposition. The Lie algebra of $G$, can be identified with $\Vect M$, which is the space of all smooth vector fields on $M$. The identification can by made by associating an equivalence class of curves $[t \mapsto \gamma(t)] \in T_1G$ to the vector field
$$Xf(m) = \left. \frac{d}{dt} f(\gamma(t)) \right|_{t=0}, \quad \gamma(0)=m,\quad f \in C^\infty(M).$$
The Lie brackets are the negative to the usual commutator brackets of vector fields on $M$, see also~\cite{Milnor}. Let $\Diff_0 M$ denote the identity component of $\Diff M$. Then, it is possible to prove controllability on $\Diff_0 M$ with respect to an invariant horizontal sub-bundle by showing that the Lie sub-algebra is bracket generating on~$M$.

\begin{theorem}[\cite{AC}]\label{theorem:AC}
Let $M$ be a compact manifold and let $\Lieh$ be a subspace of $\Vect M$, which is also a $C^\infty(M)$-sub-module. Let $\calH$ be the horizontal distribution on $\Diff M$ obtained by left (or right) translation of $\Lieh$.
If $\Lieh$ is bracket generating on $M$, i.e., if $\Lie_m \Lieh = T_mM$ for any $m \in M$, then any two diffeomorphisms $\phi_1, \phi_2 \in \Diff M$ can be connected by an $\calH$-horizontal curve.
\end{theorem}
In particular, if $\Lieh$ consists of all sections in a bracket generating sub-bundle $\calE$ of $TM$, then we have complete controllability with respect to $\calH$. Remark that Dusa McDuff communicated a similar statement to John Milnor  earlier in 1984, see~\cite[page 1018]{Milnor}.

\section{Infinite-dimensional Lie groups with constraints}\label{Lie group}

\subsection{Regular Lie groups}
Let $G$ be a Lie group modeled on a convenient vector space with the Lie algebra $\Lieg$. We use the symbol $\ell_a$ to denote the left multiplication by an element $a\in G$. Let us define the {\it left Maurer-Cartan form} $\kappa^\ell$ by  the formula
$$\kappa^{\ell}(v) = d \ell_{a^{-1}} v, \qquad v \in T_aG.$$
The Maurer-Cartan form is a $\Lieg$-valued one-form on $G$. 
For any smooth curve $\gamma\colon \real\to G$ we associate a smooth curve $u(t) = \kappa^\ell(\dot \gamma(t))$, $t\in \real$, in the Lie algebra $\Lieg$ which is called {\it the left logarithmic derivative} of $\gamma$.  All groups possessing the converse property, i.e., any curve $u \in C^\infty(\real, \Lieg)$ can be integrated to a smooth curve in $G$, have gained a special interest. More precisely, we have the following definition.

\begin{definition} \cite{KMGroup,Milnor} \label{def:regularLie}
A Lie group $G$ is called regular if
\begin{itemize}
\item[(a)] any smooth curve $u \in C^\infty(\real, \Lieg)$, is the left logarithmic derivative of some curve $\gamma\colon \real\to G$, starting at the identity $\mathbf1\in G$;
\item[(b)] the mapping
$$\begin{array}{ccc} C^\infty(\real, \Lieg) & \to & G \\ 
{[t \mapsto u(t)]} & \mapsto  & \gamma(1) 
\end{array}
$$
is smooth. Here $\gamma$ is a solution to the equation $\kappa^\ell(\dot \gamma(t)) = u(t)$, $t\in \real$ with the initial condition $\gamma(0) = \mathbf 1$.
\end{itemize}
\end{definition}
Throughout the paper, all mentioned Lie groups  are assumed to be regular. So far, there has been no known  examples of non-regular Lie groups. The term `regular' has been also used for somewhat stricter conditions, see~\cite{LieFre}. 

Let us notice the following properties of regular Lie groups.
\begin{itemize}
\item For any Lie group, not necessarily regular, a solution to the initial value problem
\begin{equation} \label{eq:solutionUn} \kappa^\ell(\dot \gamma(t)) = u(t), \qquad \gamma(0) = a,\end{equation}
 is unique. Hence the mapping in Definition \ref{def:regularLie} (b) is well defined. Clearly, (a) holds, if and only if, \eqref{eq:solutionUn} always has a solution, because we can use left multiplication by $a$ in order to let the solution to start from the identity.
\item Identifying elements $\Lieg$ with the constant curves in $C^\infty(\real,\Lieg)$, the smooth exponential map $\exp_G\colon \Lieg \to G$ in regular Lie groups is given by (b). However, the exponential map is not necessarily locally surjective, and it does not need to satisfy the Baker-Campbell-Hausdorff formula.
\item Regularity of a Lie group can be similarly defined in terms of the right logarithmic derivative. Let $r_a$ denote the right multiplication by $a$, and let $\kappa^r(v) = d r_{a^{-1}} v$, $v \in T_aG$ be the right Maurer-Cartan form.  Then for a given $\gamma\colon \real\to G$, the curve $u(t) = \kappa^r(\dot \gamma(t))$, $t\in \real$, is called the right logarithmic derivative. In this case regularity of the group implies uniqueness of the solution to the initial value problem $\kappa^r(\dot \gamma(t)) = u(t), \gamma(0) = \mathbf 1$. The property of a group to be regular does not depend on the choice between left or right logarithmic derivatives in the definition. 
\end{itemize}

\subsection{Sub-Riemannian geodesics on regular Lie groups} \label{sec:GroupCritical}

In this section we define the left-invariant sub-Riemannian structure on a regular Lie group and study the set of critical points of the energy functional defined by a sub-Riemannian metric.

Let $G$ be a regular Lie group with the Lie algebra $\Lieg$ on which an inner product  $\langle \cdot , \cdot \rangle$ is defined. Let $\mathbf g$ be a left-invariant metric on $G$ corresponding to the inner product:
$$\mathbf g(v_1,v_2) = \langle \kappa^\ell(v_1), \kappa^\ell(v_2) \rangle,\qquad  v_1,v_2\in TG.$$
Choose a $c^\infty$-closed subspace $\Lieh$ of $\Lieg$, such that there exists another $c^\infty$-closed subspace $\Liek$ satisfying $\Lieg = \Lieh \oplus \Liek.$ See, e.g., \cite{KrieglMichor} for the definition  of a $c^\infty$-topology. Then we define splitting sub-bundle $\calH$ of $TG$ by the left translations of $\Lieh$. Notice that $v \in \calH$, if and only if, $\kappa^\ell(v) \in \Lieh$. Denote by $\mathbf h$ the restriction of the metric $\mathbf g$ to the sub-bundle $\calH$. The pair $(\calH,\mathbf h)$ will be a left-invariant sub-Riemannian structure on the Lie group $G$. 

This structure fits well to the formalism of Section~\ref{sec:InfiniteMF} with $V = \Lieg$ and $\theta = \kappa^\ell$. Indeed, since
\begin{equation} \label{eq:leftMaurerCartan} 
d\kappa^\ell(v_1, v_2) = -\left[\kappa^\ell(v_1), \kappa^\ell(v_2)\right],
\end{equation}
the corresponding map $a^\top$ exists, if and only if,  the map $\ad_x\colon y \mapsto [x,y]$ has an adjoint for each $x \in \Lieg$. The existence of the adjoint map is non-trivial in infinite dimensions, and we have to assume it in order to let the condition (A) hold.

The assumption (B) holds for any regular Lie group. To show this we define a mapping
\begin{equation} \label{tauu} 
\tau_u: C^\infty(I, \Lieg) \to C^\infty(I, \Lieg), \qquad \tau_u(x) = \dot x + [u,x],
\end{equation}
for any $u \in C^\infty(I,\Lieg)$. Then we have the identity $\Xi_\gamma (Z) = \tau_{\kappa^\ell(\dot \gamma)} \big(\kappa^\ell(Z)\big)$ by~\eqref{eq:leftMaurerCartan}. The assumption (B) follows from  the following lemma.

\begin{lemma} \label{lemma:insert0}
For any $y \in C^\infty(I,\Lieg)$, there exists a unique solution to 
\begin{equation} \label{eq:inverttau} 
\tau_u x = y, \qquad x(0) = 0.
\end{equation}
\end{lemma}

\begin{proof}
Let us construct the inverse map in order to prove this statement. Let $\gamma$ be a curve in $G$ with the left logarithmic derivative  $u$. Let $\Ad$ be the adjoint action of $G$ on $\mathfrak g$. Changing variables $x=\Ad_{\gamma^{-1}}w$, and using the formula $\partial_t\big(\Ad_{\gamma^{-1}}w\big)=\Ad_{\gamma^{-1}}\dot w-[u,\Ad_{\gamma^{-1}}w]$, we rewrite equation~\eqref{eq:inverttau} in the form $\dot w=\Ad_{\gamma}y$. Solving the latter equation with the initial data $w(0)=0$, we obtain $w(t)=\Ad_{\gamma(t)}x(t)=\int_0^{t} \Ad_{\gamma(\tilde t)} y(\tilde t) \, d\tilde t$. The integral always exists in the convenient vector space $C^\infty(I, \Lieg)$.
Thus, the inverse to the map~\eqref{tauu} with the domain restricted to the subspace of functions with $x(0) = 0$, is given by 
$$x(t)= \Ad_{\gamma(t)^{-1}} \int_0^t \Ad_{\gamma(\tilde t)} y(\tilde t) \, d\tilde t.$$
\end{proof}

Write $\ad^\top_x$ for the adjoint of $\ad_x$, that is, the map satisfying $\langle [x,y_1], y_2 \rangle = \langle y_1, \ad_x^\top y_2 \rangle.$ Then the results of Section \ref{sec:InfiniteMF} can be reformulated  for regular Lie groups.

\begin{theorem} \label{GeoGroup}
Let $G$ be a regular Lie group with the Lie algebra $\Lieg$. Assume that there is an inner product $\langle \cdot, \cdot \rangle$ on $\Lieg$, such that  the adjoint $\ad_x^\top$ is well defined for any $x \in \Lieg$. Let $\Lieg = \Lieh \oplus \Liek$ define a splitting of $\Lieg$ into two $c^\infty$-closed subspaces.

Define a sub-bundle $\calH$ by left translations of $\Lieh$, and a metric on $\calH$ by 
$$\mathbf h(v_1, v_2) = \langle \kappa^\ell(v_1), \kappa^\ell(v_2) \rangle \qquad v_1, v_2 \in \calH_m.$$
Then a sub-Riemannian geodesic $\gamma$ is either semi-rigid or a normal. In the latter case the curve $\gamma$ is a solution to
\begin{equation} \label{eq:NGGroup} u = \kappa(\dot \gamma), \qquad \dot u = \pr_{\Lieh} \ad_u^\top(u + \lambda), \qquad \dot \lambda = \pr_{\Liek} \ad^\top_u(u+\lambda).\end{equation}
All solutions to \eqref{eq:NGGroup} are sub-Riemannian geodesics.
\end{theorem}

\begin{proof}
Conditions (A) and (B) are satisfied by the hypothesis of the theorem. The rest is the consequence of Theorem~\ref{theorem:main} and the fact that $a^\top(x,y) = - \ad^\top_xy$.
\end{proof}

\begin{remark}
 For the special case $\Lieh = \Lieg$, the equation \eqref{eq:NGGroup} becomes the left Euler-Poincar\'e-Arnold equation on $G$, see~\cite{Arnold}.

The same arguments can be used for a right-invariant sub-Riemannian structure. The relations $d\kappa^r(v_1, v_2) = [\kappa^r(v_1), \kappa^r(v_2)]$ lead to $a^\top(x,y) = \ad_x^\top(y)$ in this case.
\end{remark}

\subsection{Semi-rigid curves and regular Lie groups}

As it was mentioned before, a disadvantage of the definition of a semi-rigid curve is that it is hard to prove the existence of such curves and to find explicit formulas for them. However, for a regular Lie group with a left- (or right-) invariant distribution one can restrict the search to curves in the Lie algebra. We consider only the left-invariant case, because the right-invariant case is analogous.

Let $G$ be a Lie group with the Lie algebra $\Lieg$, and let $\Lieg = \Lieh \oplus \Liek$ be a splitting into closed subspaces. Define $\calH$ as a left-invariant distribution corresponding to $\Lieh$. Let $\gamma$ be a curve with the left logarithmic derivative $u$. Then a vector field $Z$ along $\gamma$ is in $\Vect_{\calH}(\gamma)$, if and only if, $z = \kappa^\ell(Z)$ satisfies the equation $\pr_{\Liek} \tau_u z = 0,$ where $\tau_u$ is as in~\eqref{tauu}. Hence, a curve $\gamma$ is semi-rigid if there is some curve $z \in C^\infty(I, \Lieg)$ satisfying $z(0) = 0, z(1) = 0$ and $\pr_{\Liek} \tau_u z =0$, and which does not come from any variation. The following lemma, which is a slight reformulation of a result in~\cite{Milnor}, permit us to describe the above property in terms of curves in the Lie algebra only.

\begin{lemma} \cite[Lemma 8.8]{Milnor}  \label{lemma:Milnor} 
Let $G$ be a regular Lie group with the Lie algebra $\Lieg$.
Let us consider two elements in $C^\infty(I \times (-\epsilon, \epsilon), \Lieg)$ given by
$$(t,s)  \mapsto  u^s(t) \qquad\text{and} \qquad (t,s) \mapsto z^s(t).$$
Then a solution $\gamma^s(t)\colon I \times (-\epsilon, \epsilon)\to G$ to the system  of differential equations 
$$u^s(t) = \kappa^{\ell}( \dot \gamma^s(t)), \qquad z^s(t) = \kappa^{\ell}( \partial_s \gamma^s(t)),$$
exists, if and only if, the functions $u$ and $z$ satisfy the condition
\begin{equation}\label{integrability}
\partial_s u^s(t) = \tau_{(u^s(t))} z^s(t)=\dot z^s(t)+[u^s(t),z^s(t)],\quad t\in I,\ \ s\in(-\epsilon, \epsilon).\end{equation}
\end{lemma}

\begin{proposition}
Let  us consider an $\calH$-horizontal curve $\gamma\colon I\to G$ with a left logarithmic derivative $u\colon I\to\Lieg$. The curve $\gamma$ is  semi-rigid if and only if there is a curve $z \in C^\infty(I,\Lieg)$ with
\begin{equation} \label{eq:zHvar} z(0) = 0, \quad z(1) = 0, \quad \pr_{\Liek} \tau_u z = 0,\end{equation}
such that the problem
\begin{equation} \label{eq:Problem}
\begin{cases}
\begin{array}{ll}
\partial_s u^s = \tau_{(u^s)} z^s, 
\\
u^s(t) \in \Lieh, &\text{for}\quad (t,s)\in I \times (-\epsilon, \epsilon)
\\
z^s(t) \in \Lieg, &\text{for}\quad (t,s)\in I \times (-\epsilon, \epsilon)
\\
u^0(t) = u(t), \quad z^0(t) = z(t), \qquad & \text{for}\quad t\in I
\\
z^s(0) = z^s(1) = 0, &\text{for}\quad s\in (-\epsilon, \epsilon)
\end{array}
\end{cases} 
\end{equation}
has no solution.
\end{proposition}

\begin{proof}
In order to prove the necessary and sufficient parts, we actually show that the curve $\gamma$ is not semi-rigid if and only if there always exists a solution to~\eqref{eq:Problem}.

Observe, that if $z$ satisfies~\eqref{eq:zHvar} and $\gamma$ is not semi-rigid, then there is an element $\gamma^s(t)\in\mathcal J_{\calH}(\gamma)$, satisfying $\kappa( \partial_s \gamma^s(t)) = z(t).$ If we  define the functions $u^s(t) = \kappa^\ell(\dot \gamma^s(t))$ and $z^s(t) = \kappa^\ell(\partial_s \gamma^s(t))$, then it is easy to see from Lemma~\ref{lemma:Milnor} that the pair $(u^s(t), z^s(t))$ satisfies all conditions of~\eqref{eq:Problem}.

Conversely, assume that for an arbitrary $z$ satisfying~\eqref{eq:zHvar}, there is a solution  $(t,s) \mapsto (u^s(t), z^s(t))$ to~\eqref{eq:Problem}. Let $\gamma^s(t)$ be a solution to 
$$u^s(t) = \kappa^\ell(\dot \gamma^s(t)), \qquad z^s(t) = \kappa^\ell(\partial_s \gamma^s(t)),
$$ which exists by Lemma~\ref{lemma:Milnor}. We choose the solution satisfying $\gamma^0(0) = \gamma(0)$, which is unique. 
Then,
\begin{itemize}
\item The condition $u^0 = u$ implies $\gamma^0(t) = \gamma(t)$;
\item $u^s(t) \in \Lieh$ yields that $\gamma^s(t)$ is $\calH$-horizontal;
\item $z^s(0) = z^s(1) = 0$ ensures that $\gamma^s(0)=\gamma(0)$ and $\gamma^s(1)=\gamma(1)$.
\end{itemize}
We conclude that $\gamma^s(t)\in\mathcal J_{\calH}(\gamma)$. Since $z$ was arbitrary, we conclude that $\gamma$ is not semi-rigid.
\end{proof}

Now we formulate the following statement about a possible existence of variations. 
\begin{proposition}
Let $\gamma\colon I\to G$ be an $\calH$-horizontal curve with a left logarithmic derivative~$u$. For any $(u,z)$ satisfying~\eqref{eq:zHvar}, there exists a pair $(u^s(t), z^s(t))$ satisfying all requirements of~\eqref{eq:Problem} except for the equality $z^s(1)=0$ for all values of $s$. 
\end{proposition}

\begin{proof}
In order to construct a pair $(u^s,z^s)$ we let $\tau_u z = v \in C^\infty(I,\Lieh)$. Define $u^s(t) = u(t) + s v(t)$ and find a curve $\gamma^s(t)$ that solves the initial value problem 
$$\kappa^\ell(\dot \gamma^s) = u^s, \qquad \gamma^s(0) = \gamma(0).$$
The solution exists because the Lie group is regular. Then we define $z^s(t) = \kappa^\ell(\partial_s \gamma^s(t))$ which  satisfies $z^s(0) = 0$, and find the curve $v(t)=\tau_{(u^s)} z^s(t)$ which belongs to $\Lieh$.
\end{proof}
In the case $\Lieh = \Lieg$, it is known that any $z \in C^\infty(I, \Lieg)$ with $z(0) = z(1) = 0$ comes from a variation given by $\gamma^s(t) = \gamma(t) \cdot \exp_G(sz(t))$.

\subsection{Geodesics for a metric invariant under a Lie subgroup} 
Up to now, the unique requirement for the subspace $\Liek$ is to form a complement to the subspace $\Lieh$. In this section, we assume that $\Lieh$ admits a complement $\Liek$, which is the Lie algebra of a connected subgroup $K$ of $G$. If $\Liek$ is finite-dimensional, then this holds if and only if $\Liek$ is a sub-algebra, but if $\Liek$ is infinite-dimensional, this is only a necessary condition.

We  study sub-Riemannian structures, which are invariant under the action of a subgroup $K$. Let us start with the left-invariant distribution $\calH$. Assume that this distribution is also right-invariant with respect to elements from $K$. This may hold if and only if $\Lieh$ is invariant under the adjoint action by $K$. We remark that if $K$ is finite-dimensional, then this holds if and only if $\Lieh$ is invariant under the adjoint action of $\Liek$, i.e., if $[\Liek, \Lieh] \subseteq \Lieh$. However, $\ad(\Liek)$-invariance of $\Lieh$ does not imply $\Ad(K)$-invariance in the case of an infinite-dimensional subgroup $K$. Under the assumption of $\Ad(K)$-invariance of $\Lieh$, the question of controllability in $G$  reduces to proving that one can reach any element in $K$ by an $\calH$-horizontal curve. This becomes a particular advantage when $K$ is a much smaller group than~$G$.

\begin{proposition} \label{groupcontrollability}
Let $G$ be a Lie group with the Lie algebra $\Lieg$, and let a left- (or right-) invariant horizontal sub-bundle $\calH$ be obtained by left (or right) translations of a subspace $\Lieh \subseteq \Lieg$. Assume that there is a sub-group $K$ of $G$ with the Lie algebra $\Liek$, such that $\Lieg = \Liep \oplus \Liek$ for some closed $\Liep \subseteq \Lieh$. Suppose also that $\Lieh$ is $\Ad(K)$-invariant.

Then any pair of elements in $G$ can be connected by a smooth $\calH$-horizontal curve, if and only if, for every $a \in K$ there is an $\calH$-horizontal smooth curve connecting $\mathbf 1$ and $a$.
\end{proposition}

\begin{proof}
We present the proof only for the case of a left-invariant sub-bundle $\calH$. Let $c\colon[0,1]\to G$ be any curve (not necessarily horizontal), connecting points $a_0$ and $a_1$, and having left logarithmic derivative $u$. Using the left translation of $c$ by $a_0^{-1}$, we can assume that $a_0 = \mathbf 1$. 

Let $\pr_{\Liek}\colon \Lieg \to \Liek$ be a projection with the kernel $\Liep \subseteq \Lieh$. Define $k(t) = \pr_{\Liek} u(t)$, and let $\vartheta$ be a curve in $K$ with a left logarithmic derivative $k$, starting at $\mathbf 1$. Then the left logarithmic derivative of the curve $\vartheta(t)^{-1}$ is $-\Ad_{\vartheta} k$. 

Let us show that the curve $\gamma_1(t)= c(t) \cdot \vartheta(t)^{-1}$ is $\calH$-horizontal. We calculate the left logarithmic derivative by
$$\kappa^\ell( \partial_t ( c(t)  \cdot \vartheta(t)^{-1})) = \Ad_{\vartheta(t)}(u(t) - k(t)) \in \Lieh.$$
Hence, we have constructed a horizontal curve $\gamma_1$, from ${\mathbf 1}$ to $a_1 \cdot \vartheta(1)^{-1}$. Applying the right translation by $\vartheta(1)$, we get a curve from $\vartheta(1)$ to $a_1$.  It keeps the curve horizontal because of the $\Ad(K)$-invariance of $\Lieh$. Moreover, by the hypothesis of the theorem, we can connect ${\mathbf 1}$ with $\vartheta(1)$ by a smooth horizontal curve $\gamma_2$. Finally, we glue the curves $\gamma_1$ and $\gamma_2$ into one smooth curve by slowing down the speed at the connecting point to zero.
\end{proof}

We continue considering the properties of the sub-Riemannian metrics coming from an invariant inner product. Let $\langle \cdot , \cdot \rangle$ be an inner product on $\Lieg$, such that $\ad_x^\top$ is well defined, and such that $\Lieh$ is orthogonal to $\Liek$. As we have done before, let $\mathbf g$ be a Riemannian metric obtained by a left translation of $\langle \cdot , \cdot \rangle$, and let us use $\mathbf h$ for its restriction to $\calH$. Then we have the following result regarding the invariance of the inner product under either $\Liek$ or $K$.

\begin{theorem} \label{TheoremLeft}
The following statements hold.
\begin{itemize}
\item[(a)] If $\langle \ , \, \rangle$ is $\ad(\Liek)$-invariant and if $\gamma_R\colon I \to G$ is a Riemannian geodesic with respect to $\mathbf g$, then $\lambda = \pr_{\Liek} \kappa^\ell(\dot \gamma_R),$
is constant. Here $\pr_{\Liek}\colon\Lieg\to\Liek$ is the orthogonal projection with respect to $\langle \ , \, \rangle$.
\item[(b)] If $\langle \ , \, \rangle$ is $\Ad(K)$-invariant, then a horizontal curve $\gamma_{sR}\colon I\to G$ is a normal sub-Riemannian geodesic with respect to $\mathbf h$, if and only if, it is of the form
\begin{equation} \label{eq:sRKinvariant} \gamma_{sR}(t) = \gamma_R(t) \cdot \exp_G(-\lambda t), \qquad \lambda(t) = \pr_{\Liek} \kappa^\ell(\dot \gamma_R(t)),\quad t\in I,
\end{equation}
where $\gamma_R\colon I\to G$ is a Riemannian geodesic with respect to $\mathbf g$.
\end{itemize}
\end{theorem}

\begin{proof}
{\bf (a)} Let us recall that the $\ad(\Liek)$ invariance of the inner product $\langle \cdot , \cdot \rangle$ means that $\langle \ad_x y_1, y_2 \rangle=-\langle y_1, \ad_x y_2 \rangle.$ Let $u_R$ be the left logarithmic derivative of $\gamma_R$. Then $u_R$ is a solution to $\dot u_R = \ad_{u_R}^{\top} u_R$ by Theorem~\ref{GeoGroup}. This implies that $\langle \dot{u}_R, k \rangle = \langle u_R , [u_R, k] \rangle = 0$ for an arbitrary $k \in \Liek$ by the $\ad(\Liek)$-invariance. Since $\dot \lambda=\pr_{\Liek}\dot u_R$, we conclude that $\lambda$ is constant.

\noindent
{\bf (b)} 
Let us remark first that since the metric is $\Ad(K)$-invariant, the orthogonality of $\Liek$ and $\Lieh$ along with the obvious invariance of $\Liek$ under the adjoint action of $K$ imply that $\Lieh$ is invariant under the adjoint action of $K$ as well. In its turn, this implies $[\Liek, \Lieh] \subseteq \Lieh.$

Let $\gamma_{sR}$ be a sub-Riemannain geodesic with the left logarithmic derivative~$u_{sR}$. Then it is a solution to the equations
\begin{equation}\label{eq:sR}
\dot u_{sR}  =\pr_{\Lieh} \ad_{u_{sR}}^{\top}(u_{sR} +\lambda),\quad\dot\lambda=\pr_{\Liek} \ad_{u_{sR}}^{\top}(u_{sR} +\lambda).
\end{equation}
Since for any $k \in \Liek$, we have
$$\langle \dot \lambda, k \rangle = \langle u_{sR}, [u_{sR}, k] \rangle + \langle \lambda, [u_{sR}, k] \rangle = 0,$$
we conclude that $\dot\lambda=0$. The last product is equal to $0$ by the identity $[\Liek, \Lieh] \subseteq \Lieh$. It follows that $\ad_{u_{sR}}^{\top}(u_{sR} +\lambda)$ is a curve in $\Lieh$ and the equations~\eqref{eq:sR} reduce to 
\begin{equation}\label{eq:sRsimple}
\dot u_{sR}  = \ad_{u_{sR}}^{\top}(u_{sR} +\lambda).
\end{equation}
We need to show that equation~\eqref{eq:sRsimple} holds if and only if $\gamma_{sR}$ is of the form~\eqref{eq:sRKinvariant}. 

Let us assume that~\eqref{eq:sRKinvariant} holds, where $\gamma_R$ is a Riemannian geodesic with the left logarithmic derivative $u_{R}$ satisfying the geodesic equation $\dot u_{R}  = \ad_{u_{R}}^{\top}(u_{R})$, and such that $\lambda=\pr_{\Liek} \kappa^\ell(\dot \gamma_R(t))$.  
Notice that the left logarithmic derivative of the right hand side of~\eqref{eq:sRKinvariant} is $$u_{sR}(t) = \Ad_{\exp_G(\lambda t)}(u_R(t) - \lambda).$$ 
Then we have the following chain of equalities for any $x \in \Lieg$:
\begin{align*} \langle \dot u_{sR}, x \rangle
& =  \left\langle\Ad_{\exp_G(\lambda t)}\big( [\lambda, u_{R} -\lambda] + \dot u_R \big), x \right\rangle \\
& =  \left\langle [\lambda, u_R] + \dot u_R , \Ad_{\exp_G(- \lambda t)}(x) \right\rangle  & 
\text{(by $\Ad(K)$ invariant metric)}
\\
& =  \left\langle [\lambda, u_R] + \ad_{u_R}^{\top}( u_R) , \Ad_{\exp_G(- \lambda t)}(x) \right\rangle  & 
\text{(by geodesic equation)}
\\
& =  \left\langle [\lambda, u_R]  , \Ad_{\exp_G(- \lambda t)}(x) \right\rangle
+ \left\langle  u_R , \big[u_R, \Ad_{\exp_G(- \lambda t)}(x)\big] \right\rangle  
\\
& = \left\langle  u_R , \big[u_R -\lambda, \Ad_{\exp_G(- \lambda t)}(x)\big] \right\rangle  & 
\text{(by $\ad(\Liek)$ invariant metric)}
\\
& = \left\langle \Ad_{\exp_G( \lambda t)} u_R , \big[\Ad_{\exp_G( \lambda t)}(u_R -\lambda), x\big] \right\rangle  & 
\text{(by $\Ad(K)$ invariant metric)}
\\
& = \left\langle u_{sR} , [u_{sR}, x] \right\rangle +  \left\langle \lambda, [u_{sR}, x] \right\rangle = \left\langle \ad_{u_{sR}}^{\top} (u_{sR} + \lambda), x \right\rangle. 
\end{align*}
We conclude that $u_{sR}$ satisfies  equation~\eqref{eq:sRsimple}.

Conversely, suppose that $u_{sR}$ is a solution to the equation~\eqref{eq:sRsimple}. Similarly, we show that $u_R(t) = \Ad_{\exp_G(-\lambda t)}(u_{sR} + \lambda)$ with constant $\lambda$ is a solution to the equation $\dot u_R = \ad_{u_R}^{\top}(u_R)$. Thus, the equations $u_R(t) = \Ad_{\exp_G(-\lambda t)}(u_{sR} + \lambda)$ and $\dot u_R = \ad_{u_R}^{\top}(u_R)$  define uniquely the Riemannian geodesic $\gamma_{R}$ of the form $\gamma_{R}(t)=\gamma_{sR}(t)\exp_G(\lambda t)$, $t\in[0,1]$. It implies~\eqref{eq:sRKinvariant}. 
\end{proof}

We emphasize the following fact. 
 \begin{corollary}\label{cor1}
The left logarithmic derivative $u_{sR}$ of a curve $\gamma_{sR}$ satisfies the equation $\dot u_{sR} = \ad_{u_{sR}}^{\top}(u_{sR} + \lambda)$ with a constant $\lambda \in \Liek$.
 \end{corollary}

\begin{remark}
Theorem~\ref{TheoremLeft}(b) is still valid when $\mathbf g$ is a pseudo-metric (not necessarily positively definite) that restricts to a positively definite metric along $\calH$. An important information is that $\Lieh$ and $\Liek$ remain orthogonal, satisfy $\Lieg = \Lieh \oplus \Liek$, and the map $\ad^\top_X$ exists with respect $\langle \cdot , \cdot \rangle$. The geodesic $\gamma_R$ is no longer  Riemannian but still is a critical curve of $E(\gamma) = \int_0^1 \mathbf g(\dot \gamma, \dot \gamma) \, dt$.

Theorem~\ref{TheoremLeft}(b) can be generalized to principal bundles in the case of finite-dimensional manifolds, see~\cite[Theorem 11.8]{Mon}.
\end{remark}

\section{Group of diffeomorphisms on the unit circle and the Virasoro-Bott group} \label{sec:DiffVir}

In this section we apply previous results to  two concrete examples of infinite-dimensional geometry with constraints. They are two  infinite-dimensional Lie-Fr\'echet groups:  the group of diffeomorphisms of the unit circle and its central extension known as the Virasoro-Bott group. We first define them,  give reasons for the chosen constraints, and calculate  normal geodesics with respect to certain metrics. We show that the Euler equations for geodesics turn out to be some known non-linear PDE, namely, 
analogues of the KdV, Burgers,  Camassa-Holm and Hunter-Saxton equations. Finally, we prove the controllability on these groups directly.

Let $\theta$ either be an element of $\real$, or an element in the Lie group $S^1$ which we identify with $\real/(2\pi \integer)$. The derivatives with respect to $\theta$, will be denoted by  prime, i.e., $x' = \partial_\theta x.$

\subsection{The group $\Diff S^1$ and its Lie algebra}
Let $\Diff S^1$ denote the group of  orientation preserving diffeomorphisms of the unit circle $S^1$, which is the component of the identity of the group of all diffeomorphisms of $S^1$. Slightly  abusing notations, we shall use the symbol $\Diff S^1$ instead of $\Diff_0 S^1$ as in Section \ref{sec:controllability}, dropping the sub-index for the sake of simplicity. As a manifold, $\Diff S^1$ is modeled on the Fr\'echet space $\Vect S^1$ of smooth real vector fields on~$S^1$.  $\Vect S^1$ is the Lie algebra of the group $\Diff S^1$ consisting of vector fields $x \, \partial_\theta, x \in C^\infty(S^1)$, with the Lie brackets $[x \, \partial_\theta,y \, \partial_\theta]= (x'y-y'x) \partial_\theta$. Using the identification between $\Vect S^1$ and $C^\infty(S^1)$ we denote a vector field $x \, \partial_\theta$ simply as $x$. We can not use the exponential map to construct charts,  because it is not locally surjective \cite{K, Milnor}.  See also \cite{Milnor} for a description of the manifold structure on the diffeomorphism groups. In particular, the group $\Diff S^1$ is simple and non-real analytic.

We  denote by $\id$ the identity in $\Diff S^1$. Let us identify $T\Diff S^1$ and $\Diff S^1 \times \Vect S^1$    by associating the element $(\gamma(0), \dot \gamma(0) \partial_{\theta})$ to the equivalence class of curves $[t  \mapsto \gamma(t)] \in T_{\gamma(0)} \Diff S^1$ passing through $\gamma(0)$.
The left and right actions can then be described as 
\begin{equation} \label{leftrightDiff} d\ell_{\varphi} (\phi, x \partial_\theta) = \big(\varphi \circ \phi, (\varphi' x) \partial_\theta \big), \qquad
dr_{\varphi} (\phi, x \partial_\theta) = \big(\phi \circ \varphi, (x \circ \varphi) \partial_\theta\big),\end{equation}
where $\phi, \varphi \in \Diff S^1, x \in C^\infty(S^1)$.  Notice that  \eqref{leftrightDiff} implies $\Ad_{\varphi} x \partial_{\theta} = \varphi' x(\varphi^{-1}) \partial_\theta.$

It is often convenient to work with the universal covering group $\wDiff S^1$ of $\Diff S^1$ that consists of all orientation preserving diffeomorphisms of $\real$, satisfying $\phi(\theta + 2\pi) = \phi(\theta) + 2 \pi.$

\subsection{The Virasoro-Bott group} \label{sec:Vir-Bott}

The group $\wDiff S^1$ has a unique non-trivial central extension by $\real$ called the {\it Virasoro-Bott group}.
It can be described as follows. Define a Lie algebra $\Lieg_{\mu\nu}$ as the vector space $\Vect S^1 \oplus \real$, with the commutator
$$\big[(x ,a_1), (y , a_2)\big]  = \Big([x, y ], \omega_{\mu\nu}(x, y )\Big), \quad 
\omega_{\mu\nu}(x,y) = \frac{1}{2\pi}\int_0^{2\pi} \left(\mu x(\theta) y'(\theta) + \nu x'(\theta) y''(\theta)\right)\, d\theta.$$
The extension is trivial if and only if $\nu= 0$ since the term $\omega_{1,0}(x,y)=\frac{1}{2\pi}\int_0^{2\pi} x(\theta) y'(\theta)\, d\theta $ represents the algebra 2-coboundary. Indeed, let us introduce a linear map $\eta\colon\Vect S^1\to\mathbb R$ by
\begin{equation}\label{eta}
\eta(x \partial_\theta) = \frac{1}{2\pi} \int_0^{2\pi} x(\theta)\,d\theta,
\end{equation}
and let us observe that $\omega_{10}(x,y)=-\frac{1}{2}\eta([x,y])$.  All non-trivial extensions coinciding modulo an algebra 2-coboundary are isomorphic Lie algebras~\cite{GelFu}. Hence, the Lie algebras $\Lieg_{\mu\nu}$ with $\nu \neq 0$ are all isomorphic, because they differ only by a 2-coboundary up to a scaling factor. The extended Lie-Frech\'et algebra $\Lieg_{\mu\nu}$ for $\nu \neq 0$ is called the {\it Virasoro algebra}. This attribution appeared because of in physics~\cite{Virasoro}. The algebra 2-cocycle $\omega_{\mu\nu}$ is called the Gelfand-Fuchs cocycle. There is a simply connected Lie group $\calG_{\mu\nu}$ corresponding to each Lie algebra $\Lieg_{\mu\nu}$. It can be considered as the set $\wDiff S^1 \times \real$ with the group operation
\begin{equation}\label{multVir}
(\phi_1, b_1) (\phi_2,b_2) = \Big(\phi_1 \circ \phi_2, b_1+ b_2 + \mu A(\phi_1, \phi_2) + \nu B(\phi_1,\phi_2)\Big),
\end{equation}
where
$$A(\phi_1, \phi_2) = \frac{1}{4\pi} \int_0^{2\pi} (-\phi_1 \circ \phi_2  + \phi_1 + \phi_2 - \id) d\theta,\quad\id\in\wDiff S^1,$$
$$B(\phi_1, \phi_2) = \frac{1}{4\pi} \int_0^{2\pi} \log (\phi_1 \circ \phi_2)' d \log \phi_2'.$$
The group $\calG_{\mu0}$ is isomorphic to the product group $\wDiff S^1 \times \real$, where the sign $(\times)$  means the direct product of groups, while for $\nu \neq 0$, the extension $\calG_{\mu\nu}$ is non-trivial. The term $B$ in the multiplication law~\eqref{multVir} represents a Bott 2-cocycle obtained in~\cite{Bott} and the part $A$ is the group 2-coboundary. 
All the groups $\calG_{\mu\nu}$ with $\nu \neq 0$ are isomorphic and called the {\it Virasoro-Bott group} because of the Bott cocycle $B(\phi_1, \phi_2)$. The construction of the Bott cocycle is widely presented in the literature, see~\cite{Bott,K}. To find the coboundary $A$, we observe that a 2-coboundary is defined by a smooth function $F\colon\wDiff S^1\to\real$ such that 
$$
A(\phi_1, \phi_2)=F(\phi_1)+F(\phi_2)-F(\phi_1\circ\phi_2),\quad \phi_1,\phi_2\in\wDiff S^1.
$$
The existence of the identity and inverse element in $\calG_{\alpha\beta}$ implies the conditions 
$$
F(\id)=0,\qquad F(\phi)+F(\phi^{-1})=0,\quad \id,\phi\in\wDiff S^1.
$$
At the last step, one has to check that the function $F(\phi)=\frac{1}{4\pi}\int_{0}^{2\pi}\big(\phi(\theta)-\theta\big)\,d\theta$ satisfies these conditions and the infinitesimal version of the group coboundary $A$ coincides with the algebra coboundary $\omega_{10}$.

Remark, that if $(\gamma(t), b(t))$ is a curve in $\calG_{\mu\nu}$, then
$$\kappa^\ell(\dot \gamma(t), \dot b(t)) = \left( u(t) \, , \, C(t) \right),$$
where
$$u(t) = \frac{\dot \gamma(t)}{\gamma'(t)} \text{ and } C(t) = \dot b(t) - \frac{\mu}{4\pi} \int_0^{2\pi} u(t) d\theta
+ \frac{\mu}{4\pi} \int_0^{2\pi} u(t) d\gamma(t) + \frac{\nu}{4\pi} \int_0^{2\pi} u'(t) d\log \gamma'(t).$$
Here we used the formula $\kappa^\ell(\dot \gamma(t), \dot b(t))=\frac{d}{ds}\Big\vert_{s=0}(\gamma^{-1}_{t},-b(t))\cdot (\gamma_{t+s},b(t+s))$.

\subsection{Horizontal sub-bundles, CR-structure and complex structure} \label{sec:HCRC}

Notice that the linear map $\eta$ from~\eqref{eta} associates to each vector field from $\Vect S^1$ its mean value on the circle.
Let $\Vect_0 S^1$ denote the kernel of $\eta$ consisting of all vector fields with zero mean value. We use $\Liek$ to denote the subalgebra of $\Vect S^1$ of constant vector fields. Clearly $\Vect S^1 = \Vect_0 S^1 \oplus \Liek.$ The subgroup corresponding to $\Liek$  in $\Diff S^1$ is the abelian group of rotations $K= \Rot S^1\simeq S^1$.  It  corresponds to the group of translations on the universal cover, which we  denote by $\widetilde K$.

Define a horizontal sub-bundle $\calH$ of $T\Diff S^1$ or $T\wDiff S^1$ by left translations of $\Vect_0 S^1$. There exists a left-invariant almost-complex structure $J$ on $\calH$, given at the identity by the  Hilbert transform as
\begin{equation} \label{eq:HilbertT}Jx(\theta) = \frac{1}{2\pi} \text{p.v.} \int_0^{2\pi} \frac{x(t)}{\tan\left(\frac{t-\theta}{2} \right)} dt, \qquad x \in \Lieh.
\end{equation}
The triple $(\Diff S^1, \calH, J)$ (and hence, also $(\wDiff S^1, \calH, J)$) is an infinite-dimensional CR-manifold \cite{Lempert}. 
A curve $\gamma\colon I\to T\Diff S^1\ (T\wDiff S^1)$ is $\calH$-horizontal if
$$\eta\big(\kappa^\ell(\dot \gamma(t))\big)=\eta\left(\tfrac{\dot \gamma(t)}{\gamma'(t)}\right) = 0 \text{ for every } t\in I.$$

Similarly, we define a horizontal sub-bundle $\calE$ of $T\calG_{\mu\nu}$ by identifying $\Vect_0 S^1$ with a subset $(\Vect_0 S^1,0)$ of the extended algebra $\Lieg_{\mu\nu}$ and defining $\calE$ by left translations of $\Vect_0 S^1$ on $\calG_{\mu\nu}$. There is a complex structure on $\calG_{\mu\nu}$, such that $\calE$ becomes a holomorphic vector bundle~\cite{Lempert}. The complex structure on $\calG_{\mu\nu}$ restricted to $\calE$ is also obtained by left translation of the  Hilbert transform \eqref{eq:HilbertT}, and we denote it by the same symbol $J$. A choice of complement of $\Vect_0 S^1$ in $\Lieg_{\mu\nu}$ is given by
$$\widehat \Liek = \{ (a_0 \partial_\theta, a) \in \Lieg_{\mu\nu} \, : \, a_0,a \in \real \}.$$
This is an abelian sub-algebra corresponding to the abelian sub-group
$$\widehat K = \{(\theta \mapsto \theta + b_0, b) \in \calG_{\mu\nu} \, : \, b_0, b \inÊ\real \}.$$

\begin{proposition}\label{invAdH}
The sub-bundle $\calH$ of $T\Diff S^1$ {\rm(}or $T\wDiff S^1${\rm)} is invariant under the action of rotations $K$ {\rm(}or translations $\widetilde K${\rm)}, and the sub-bundle $\calE$ of $T\calG_{\alpha\beta}$  is invariant under the action of~$\widehat K$.
\end{proposition}
\begin{proof}
If $\rho: \theta \to \theta + b_0$ is a rotation/translation, then $\Ad_{\rho}(x)(\theta) = x(\theta - b_0)$.  Therefore, 
$\eta(\Ad_{\rho}(x)) = \eta(x)$, which means that $\calH$ is invariant under the action of $K$ (or $\widetilde K$). By  similar arguments $\calE$ is invariant under $\widehat K$.
\end{proof}

As a corollary of the proof of Proposition~\ref{invAdH} and~\eqref{eq:HilbertT}, we obtain
\begin{equation} \label{eq:Jinvariant} J\Ad_{\rho}(x) = \Ad_{\rho}(Jx).\end{equation}
for  $\rho\in K, \widetilde K$ or $\widehat K$ and $ x\in \calH$ or $x\in \calE$ respectively.

\subsection{Normal geodesics with respect to the sub-Riemannian metrics}
Let us describe the normal sub-Riemannian geodesics with respect to a two-parameter family of left-invariant metrics on $\calH$ and $\calE$, that includes the Sobolev $H^0, H^1$, and $H^{1,1}$ metrics.

Let $\langle \ , \, \rangle^{1,0}$ denote the standard $L^2$ or $H^0$ inner product on $\Vect S^1$
$$\left\langle x , y \right\rangle^{1,0} = \frac{1}{2\pi} \int_0^{2\pi} x(\theta) y(\theta) d\theta.$$
Observe that  $\ad_x^{\top}$ exists with respect to this inner product and it is given by the formula
\begin{equation} \label{eq:L2coadjoint} 
\ad_x^{\top} y = x y' + 2 x' y.
\end{equation}
Let $\mathbf g^{1,0}$ be the Riemannian metric obtained by the left translation of $\langle \cdot , \cdot \rangle^{1,0}$, and let $\mathbf h^{1,0}$ be its restriction to $\calH$. If we denote by $p_0 = \partial_\theta$ the basis vector for $\Liek$, then
$$
\langle \ad_{p_0} x , y \rangle^{1,0}=\langle [p_0, x ] , y \rangle^{1,0} = - \langle x' , y \rangle^{1,0} = \langle x , y' \rangle^{1,0}= - \langle x , [p_0,y] \rangle^{1,0}=-\langle x ,\ad_{ p_0}y \rangle^{1,0}.
$$
This implies that the inner product $\langle \ , \, \rangle^{1,0}$ is invariant under the adjoint action of the group~$K$(or $\widetilde K$ if we are working on the universal cover). Moreover, the subspaces $\Vect_0 S^1$ and $\Liek$ are orthogonal with respect to the inner product $\langle \ , \, \rangle^{1,0}$, converting the linear map $\eta$ in~\eqref{eta} to an orthogonal projection to $\Liek$. Then, we can obtain the normal sub-Riemannian geodesics for $\mathbf h^{1,0}$ from the Riemannian geodesics for $\mathbf g^{1,0}$ by Theorem~\ref{TheoremLeft}.  The Riemannian geodesics are obtained as solutions to the Burges equation $\dot u = \ad_u^\top u = 3 u u'$, and the corresponding normal sub-Riemannian geodesics are solution to the equations
$$
\kappa^{\ell}(\dot \gamma) = u \in \Vect_0 S^1, 
\qquad\dot u = \ad_{u}^\top(u+ \lambda) = 3 u u' + 2\lambda u',\quad u\in\Vect_0 S^1,\ \lambda \in \real$$
by Corollary~\ref{cor1}. If we denote by $\mathbf b^{1,0}$  the Riemannian metric  on the symmetric space $B = \Diff S^1 /K$ induced by $\mathbf g^{1,0}$, as we described in Section \ref{sec:RiemannianSubmersions}, then the geodesics in $B$ with respect to $\mathbf b^{1,0}$ are given as projections of the normal sub-Riemannian geodesics $\gamma$ with $\lambda = 0$, or equivalently, by projections of Riemannian geodesics, which are horizontal to $\calH$.

More generally, we can define a two-parameter family $\langle \ , \, \rangle^{\alpha\beta}_0$ of inner products on $\Vect_0 S^1$ by the formula
\begin{eqnarray*}
\langle x , y \rangle^{\alpha\beta}_0  & = & \frac{1}{2\pi} \int_0^{2\pi} (\alpha x(\theta) y(\theta) + \beta x'(\theta) y'(\theta))d\theta
\\
& = & 
- \frac{1}{2\pi}\int_0^{2\pi}x(\theta)L_{\alpha\beta}y(\theta)d\theta=-\langle x ,L_{\alpha\beta}y  \rangle^{1,0},\quad x,y\in\Vect_0 S^1.
\end{eqnarray*}
Here we use the operator $L_{\alpha\beta}\colon\Vect S^1\to \Vect S^1$ defined by $L_{\alpha\beta}x = \beta \partial^2_\theta x - \alpha x$.
In order to make the bilinear map $\langle \cdot, \cdot \rangle_0^{\alpha\beta}$ to be a true inner product, we have to require  $\alpha \neq - n^2 \beta$, $n\in\mathbb N$ for  non-degeneracy, and $\beta \geq 0$ and $\alpha > - \beta$ for  positive definiteness. Let us extend the inner product $\langle \ , \, \rangle^{\alpha\beta}_0$ to the entire Lie algebra $\Vect S^1$ by the formula
\begin{equation} \label{eq:IPextension} \langle x , y \rangle^{\alpha\beta} = \big\langle x- \eta_x, y-\eta_y \big\rangle^{\alpha\beta}_{0}+ \eta_x \eta_y\quad x,y \in \Vect S^1, \quad \eta_x = \eta(x), \ \  \eta_y = \eta(y).
\end{equation}
Notice that
$$
\langle x , y \rangle^{\alpha\beta} = - \left\langle L_{\alpha\beta}(x-\eta_x), y-\eta_y \right\rangle^{1,0} + \eta_x \eta_y
= \left\langle - L_{\alpha\beta}(x-\eta_x) + \eta_x, y \right\rangle^{1,0} 
$$
since $\left\langle L_{\alpha\beta}(x-\eta_x), \eta_y \right\rangle^{1,0}=0$.
Let us define a Riemannian metric $\mathbf g^{\alpha\beta}$  by the left translation of $\langle \ , \, \rangle^{\alpha\beta}$, and let $\mathbf h^{\alpha\beta}$ be its restriction to $\calH$. It is easy to check that all conditions of Theorem~\ref{TheoremLeft} are satisfied.

Rather than finding an explicit formula for the adjoint of $\ad_x$ with respect to the inner product $\langle \cdot , \cdot \rangle^{\alpha\beta}$, which does exist but is complicated for $\beta \neq 0$, we use a simpler formula of the adjoint of the $L^2$ metric given in~\eqref{eq:L2coadjoint}. Together with Corollary~\ref{cor1}, this implies that any left logarithmic derivative $u$ of a normal sub-Riemannian geodesic for the metric $\mathbf h^{\alpha\beta}$ must satisfy
\begin{align*}
\langle L_{\alpha\beta} \dot u, x \rangle^{1,0} &= - \langle \dot u, x \rangle^{\alpha\beta} \\
& = - \langle u + \lambda, [u,x] \rangle^{\alpha\beta} = \langle L_{\alpha\beta}(u) - \lambda, x \rangle^{1,0} \\
&= \langle \ad_{u}^\top(L_{\alpha\beta} u- \lambda), x \rangle^{1,0}
\end{align*}
for any $x \in \Vect S^1$. Hence we have the geodesic equation $L_{\alpha\beta} \dot u = \ad_{u}^{\top}(L_{\alpha\beta}u-\lambda)$ or
$$
\beta \dot u'' - \alpha \dot u = \beta (u u''' +2u'u'')-3\alpha uu'-2\lambda u'
$$
If $\mathbf b^{\alpha\beta}$ is the Riemannian metric on $B$ induced by $\mathbf h^{\alpha\beta}$, then the Riemannian geodesics on $B$ are given as projections of the solutions with $\lambda = 0$.

We can also extend this inner product $\langle \ , \, \rangle^{\alpha\beta}$ to the Virasoro algebra $\Lieg_{\mu\nu}$. The extension is given by the formula
$$\Big\langle (x \partial_\theta, a_1), (y \partial_\theta, a_2) \Big\rangle^{\alpha\beta}_{\mu\nu}
= \langle x , y \rangle^{\alpha\beta} + a_1 a_2.$$
Let us calculate the adjoint $\ad_{(x, a)}^{\top}$ of $\ad_{(x,a)}$ with respect to the metric $\langle \ , \, \rangle^{1,0}_{\mu\nu}$. Notice that
$$
\omega_{\mu\nu}(x,y)=\frac{1}{2\pi}\int_0^{2\pi} \big(\mu x(\theta) y'(\theta) + \nu x'(\theta) y''(\theta)\big)\, d\theta=-\langle x, L_{\mu\nu}y'\rangle^{1,0}.
$$
Then we calculate
\begin{eqnarray}\label{eq:L2VirCoad} 
\Big\langle \ad_{(x,a_1)}^{\top} (y,a_2), (z, a_3) \Big\rangle^{1,0}_{\mu\nu}
& = &
\langle y , [x,z] \rangle^{1,0} -a_2\omega_{\mu\nu}(z,x)\nonumber
\\
&=&
\langle \ad_x^{\top}y , z \rangle^{1,0}+\langle z , a_2L_{\mu\nu}x' \rangle^{1,0}
\\
& = & 
\Big\langle (xy'+2x'y +  a_2L_{\mu\nu}x',0), (z, c) \Big\rangle^{1,0}_{\mu\nu} \nonumber
\end{eqnarray}
by formula~\eqref{eq:L2coadjoint}.

Let $\mathbf g^{\alpha\beta}_{\mu\nu}$ be the Riemannian metric on $\mathcal G_{\mu\nu}$ obtained by left translations of $\langle \ , \, \rangle^{\alpha\beta}_{\mu\nu}$, and let $\mathbf h^{\alpha\beta}_{\mu\nu}$ be its restriction to the sub-bundle $\calE$. Then  it is easy to verify that the conditions of Theorem~\ref{TheoremLeft} are satisfied. We start writing the geodesic equations for a particular case, namely for the metric $\langle \ , \, \rangle^{1,0}_{\mu\nu}$. The left logarithmic derivative $(u(t),0)\in (\Vect_0S^1,0)\subset\mathfrak g_{\mu\nu}$ of a normal sub-Riemannian geodesic $(\gamma,b)\colon I\to \calG_{\mu\nu}$ is a solution to the equation $(\dot u, 0) = \ad_{(u, 0)}^{\top}(u + \lambda_1, \lambda_2)$, $\lambda_1,\lambda_2\in\mathbb R$. This means that $u$ is a solution to
\begin{equation}\label{eq10}
\dot u = 3 u u' + (2 \lambda_1- \lambda_2 \mu) u' + \lambda_2 \nu u''',\qquad u\in\Vect_0 S^1.
\end{equation}
The corresponding Riemannian geodesics with respect to $\mathbf g^{1,0}_{0,1}$ satisfy the KdV equation.

In order to generalize the equation~\eqref{eq10}, we use the same arguments as above and conclude that
the normal critical curves $(\gamma(t),b(t))$, $t\in I$, in $\calG_{\mu\nu}$ with respect to $\mathbf h^{\alpha\beta}_{\mu\nu}$ have the left logarithmic derivative $(u,0)$ and they are solutions to the equation
$(L_{\alpha\beta} \dot u,0) = \ad_{(u,0)}^{\top}(L_{\alpha\beta}u-\lambda_1,-\lambda_2)$,
where the operator $\ad_{(u,0)}^{\top}$ is expressed as in~\eqref{eq:L2VirCoad}. This leads to the equation
$$L_{\alpha\beta} \dot u = u L_{\alpha\beta} u' + 2 u' L_{\alpha\beta} u - 2 \lambda_1 u' - \lambda_2 L_{\mu\nu} u',$$
whith $u \in \Vect_0 S^1$ and $\lambda_1, \lambda_2 \in \real$.

\begin{remark}
There are many results related to Riemannian geodesics for the invariant metrics on the Virasoro-Bott group. A good overview of these results can be found in \cite{K}. Here the right-invariant approach is chosen, but results differ only by a sign from the left-invariant point of view. We only mention here that the geodesic equations with respect to right-invariant metrics corresponding to $\langle \cdot , \cdot \rangle^{1,0}, \langle \cdot , \cdot \rangle^{0,1}$ and $\langle \cdot , \cdot \rangle^{1,1}$ respectively, are given by the KdV equation, the Hunter-Saxton equation, and the Camassa-Holm equation. Similarly, the geodesic equations with respect to the right-invariant metric on $\wDiff S^1$ produced by $\langle \cdot , \cdot \rangle^{1,0}$ and $\langle \cdot , \cdot \rangle^{1,1}$ are given by Burgers' equation and the non-extended Camassa-Holm equation.
\end{remark}

\subsection{Relationship to univalent functions}
Let us denote by $B$  the homogeneous space
\begin{equation} \label{eq:Base} 
B = \Diff S^1/ K \cong \wDiff S^1/\widetilde K \cong \calG_{\mu\nu}/\widehat K.
\end{equation}
The bundles $\calH$ and $\calE$ are the Ehresmann connections for the respective submersions $\Diff S^1 \to B$ and $\calG_{\mu\nu} \to B$. We know that $J$ from~\eqref{eq:Jinvariant} induces a well-defined almost-complex structure on $B$, which is in fact, a complex structure. A common way to visualize this is by identifying $B$ with the space of normalized conformal embeddings of the unit disk into $\mathbb C$, see~\cite{AM, Kirillov0,Kirillov1}.

Let us consider the space $\calA_0$ of all holomorphic functions
$$F: \unitD \to \comp, \qquad F(0) = 0,\quad\text{with}\quad \unitD=\{z:\,\,|z|<1\},$$
such that the extension of $F$ to the boundary $S^1$ is $C^\infty(\hat\unitD, \comp)$. Here, $\hat{\unitD}$ denotes the closure of $\unitD$. 
The class $\mathcal A_0$  is a complex Frech\'et vector space where the topology is defined by the seminorms
$$
\|F\|_m=\sup \{ |F^{(m)}(z)| \ \mid\ z\in\hat{\mathbb D}\},
$$ which is equivalent to the uniform convergence of all derivatives in $\hat{\mathbb D}$. The local coordinates can be defined by
 the embedding of $\calA_0$ to $\mathbb C^{\mathbb N}$ given by
$$F = \sum_{n=1}^\infty a_n z^n \mapsto (a_1, a_2, \dots ).$$
Let $\calF_0$ be a subclass of $\calA_0$ consisting of all univalent  functions $f\in\calA_0$, normalized by  $f'(0) =1$. The de Branges theorem~\cite{deBranges} yields that $\calF_0$ is contained in the bounded subset
$$1 \times \prod_{n=2}^\infty n \unitD \subseteq \comp^{\mathbb{N}}.$$

Let $\unitD_-$ be the exterior of the unit disk $\mathbb D=\mathbb D_+$. For any $f \in \calF_0$, we define a {\it matching function}  $g:\unitD_- \to \comp$, such that the image of $\mathbb D_-$ under $g$ is exactly the exterior of $f(\mathbb D_+)$, and let $g$ satisfy the normalization $g(\infty) = \infty$. Note that such $g$ exists by the Riemann mapping theorem. Since both functions $f$ and $g$ have a common boundary, $g$ also has a smooth extension to the closure $\hat{\mathbb D}_-$ of $\mathbb D_-$. Therefore, the images $g(S^1)$ and $f(S^1)$ are defined uniquely and represent the same smooth contour in $\comp$. If $g$ and $\widetilde g$ are two matching functions to $f$, then  they are related by a rotation
$$\widetilde g(\zeta) = g(\zeta w),\quad \zeta \in \unitD_-, \quad |w|=1.$$
For an arbitrarily matching function $g$ to $f \in \calF_0$ the diffeomorphism $\phi\in\Diff S^1$, given by
\begin{equation} \label{corsp} 
e^{i\phi(\theta)} = (f^{-1} \circ g)(e^{i\theta}),
\end{equation}
is uniquely defined by $f$ up to the right superposition with a rotation.  The relation \eqref{corsp} gives a holomorphic bijection
\begin{equation} \label{corsp1} B = \Diff S^1/K \cong \calF_0,\end{equation}
see  \cite{AM, Kirillov0,Kirillov1}.
The complex structure on $B$ induced by $\calF_0$ is the same as the one given by $J$.

\subsection{Metrics on $\calH$ corresponding to invariant K\"ahlerian metrics}

The left action of $\Diff S^1$ is well-defined on~$B$.  Let us choose an Hermitian metric on the base space $B$ assuming that this metric is K\"ahlerian and invariant under the action of $\Diff S^1$. All pseudo-Hermitian metrics on $B$ are included into the two-parameter family $\mathbf b_{\alpha\beta}$, see~\cite{Kirillov1, Unirreps, Kirillov4}. We will describe these metrics identifying $B$ and $\calF_0$ as in the previous section. It is sufficient to describe this metric only at $\id_{\unitD} \in \calF_0$ because at other points of $\calF_0$ the metric $\mathbf b_{\alpha\beta}$ is defined by the left action of $\Diff S^1$. Any smooth curve $f_t$ in $\calF_0$ with $f_0 = \id_{\unitD}$ can be written as
$$f_t(z) = z + t z F(z) + o(t), \qquad F \in \calA_0.$$
Hence, we can identify $T_{\id_{\unitD}} \calF_0$ with $\calA_0$ by relating $[t \mapsto f_t]$ to $F$. With this identification, $\mathbf b_{\alpha\beta}$ can be written as
\begin{align}\label{metric}
\mathbf b_{\alpha\beta}\big\vert_{\id_{\unitD}}(F_1, F_2) & = \frac{2}{\pi} \iint_{\unitD} \Big( \alpha F_1' \overline{F}_2' + \beta (z F_1')' \overline{(z F_2')'} \Big) d\sigma(z),\nonumber \\
& = 2\sum_{n=1}^\infty (\alpha n + \beta n^3) a_n \overline{b}_n,
\end{align}
where $d\sigma(z)$ is the area element and
$F_1(z) = \sum_{n=1}^\infty a_n z^n$, $F_2(z) = \sum_{n=1}^\infty b_n z^n$.
If $\alpha \neq -n^2 \beta, n \in \integer$, then the metric $\mathbf b_{\alpha\beta}$ is non-degenerating pseudo-Hermitian. Otherwise, $\mathbf b_{\alpha\beta}$ is degenerating along a distribution of complex dimension 1. Moreover, we require $\beta \geq 0$ and $- \alpha < \beta$ in order to obtain a positively definite Hermitian metric. 

It is  impossible to write the left action of $\Diff S^1$ on $\calF_0$  explicitly, therefore, 
it is not easy to describe $\mathbf b_{\alpha\beta}$ globally on $\calF_0$. However, these metrics can be lifted to the metrics on $\calH$ which are easier to study. Let $\pi: \Diff S^1 \to \calF_0$ be the canonical projection, and let us consider the injective map
$$\begin{array}{rccrc} d_{\id} \pi: & \Vect_0 S^1& \to & T_{\id_{\unitD}}\calF_0 \cong & \calA_0 \\
& x \partial_\theta & \mapsto & & F \end{array}. $$
Then the elements $F$ and $x$ are related by the formula, see~\cite{Kirillov1},
$$F(e^{i\theta}) = -\frac{i}{2}\big(x(\theta) - iJx(\theta)\big).$$
where $J$ is defined in~\eqref{eq:HilbertT}. Observe that
\begin{align*}
\mathbf b_{\alpha\beta}|_{\id_{\unitD}} (F_1,F_2) & = \frac{2}{\pi} \iint_{\unitD} \Big( \alpha F'_1 \overline{F}_2' + \beta (z F_1')' \overline{(z F_2')'} \Big) d\sigma(z) \\
& = \frac{-i}{\pi} \iint_{\unitD} \left( \alpha dF_1 \wedge d\overline{F}_2 + \beta d(zF_1') \wedge d\overline{(zF_2')} \right)\\
& = \frac{-i}{\pi} \int_{S^1} \left( \alpha F_1 d\overline{F}_2 + \beta (zF_1') d\overline{(zF_2')} \right).
\end{align*}
So, if $F_1$ and $F_2$ on the boundary coincides with respectively $-\frac{i}{2}\big(x - iJx\big)$ and $-\frac{i}{2}\big(x - iJx\big)$ for any $x, y \in \Vect_0 S^1$, we conclude that
\begin{align*}
\mathbf b_{\alpha\beta}|_{\id_{\unitD}}\big(d_{\id}\pi x , d_{\id} \pi y \big)
& = \frac{i}{4\pi} \int_{S^1} \Big( \alpha (x - iJx) \, d(y+ iJy) + \beta (x'-iJx') \, d(y'+iJy') \Big) \\
&= \frac{i}{4\pi} \int_{S^1} \Big( \alpha (x \, dy - iJx \, dy+ ix \, dJy + Jx \, dJy) 
 \\ &\qquad \qquad \qquad + \beta (x' \, dy' - iJx' \, dy'+ i x' \, dJy' + Jx' \, dJy') \Big) \\
 &= \frac{i}{4\pi} \int_{S^1} \Big( \alpha (x \, dy  + Jx \, dJy) + \beta (x' \, dy' + Jx' \, dJy') \Big) \\
 & \quad +  \frac{1}{4\pi} \int_{S^1} \Big( \alpha (Jx \, dy- x \, dJy )  + \beta (Jx' \, dy'- x' \, dJy' ) \Big) \\
 & = i \omega_{\alpha\beta}(x, y) + \omega_{\alpha\beta}(Jx, y),
\end{align*}
where $\omega_{\alpha\beta}$ is defined as in Section~\ref{sec:Vir-Bott} and in the last equation we used $\int_{S^1} x dy = \int_{S^1} Jx \, dJy$ which can be shown by  Fourier expansions.
The  inner product on $\Vect_0 S^1$ corresponding to the form $\omega_{\alpha\beta}$ is obtained by
$$(x, y)_{\alpha\beta} = \omega_{\alpha\beta}(Jx,y).$$
Observe that  
\begin{equation} \label{eq:updownJswitch} 
(x, y)_{\alpha\beta} = - \langle Jx', y \rangle^{\alpha\beta},\qquad x, y \in \Vect_0 S^1.
\end{equation}
Extend $(\ , \,)_{\alpha\beta}$ to an inner product on the whole algebra $\Vect S^1$ as in~\eqref{eq:IPextension}. Let $\mathbf g_{\alpha\beta}$ be a Riemannian metric obtained by the left translation of $(\ ,\, )_{\alpha\beta}$, and let $\mathbf h_{\alpha\beta}$ be the metric restricted to~$\calH$. We apply Theorem~\ref{TheoremLeft} and deduce that a normal critical curve $\gamma\colon I\to\Diff S^1$ is the solution to
$$\kappa^\ell(\dot \gamma) = u,\qquad L_{\alpha\beta} J\dot u' = u L_{\alpha\beta} Ju'' + 2u' L_{\alpha\beta} u' - 2 \lambda u',\ \ \lambda\in\mathbb R.$$
Here we used the property~\eqref{eq:updownJswitch} and the equation $L_{\alpha\beta}\frac{d}{dt}Ju'(t)=\ad^{\top}_u(L_{\alpha\beta}Ju'-\lambda)$.
We conclude that the geodesics for $\mathbf b_{\alpha\beta}$ can be found by solving the above equation for $\lambda = 0$, and then, projecting it to~$\calF_0$.

For $(\alpha,\beta) = (1,0)$, this is a special case of the modified  Constantin-Lax-Majda (CLM) equation. For more information, see~\cite{EKW,BBHM}, where the Riemannian geometry for the metric $\mathbf g_{1,0}$  is considered. It can be considered as the Sobolev $H^{1/2}$ metric on $\Diff_+ S^1$.

\begin{remark}
Let $\widetilde {\mathbf g}_{1,0}$ be the {\it right}-invariant metric on $\Diff S^1$ corresponding to $( \cdot, \cdot)_{1,0}$. Let $d_{\mathbf g_{1,0}}$ and $d_{\widetilde {\mathbf g}_{1,0}}$ be the Riemannian distance functions related to these metrics. In \cite{BBHM}, it was shown that the geodesic distance related to $\widetilde {\mathbf g}_{1,0}$ vanishes by showing that
$$d_{\widetilde {\mathbf g}_{1,0}}(\id, \rho) = 0, \qquad \text{ for any } \rho \in K = \Rot(S^1).$$
Because of the isomorphism $\phi \mapsto \phi^{-1}$ between the right- and left- invariant structures, this means that $d_{{\mathbf g}_{1,0}}(\id, \rho) = 0$ also  for any $\rho \in K$. However, this is different  when working with the induced metrics on the quotient spaces.

Consider the projection
$$\pi: \Diff S^1 \to B = \Diff S^1/K,$$
which we can identify with $\calF_0$. Let $\mathbf b_{1,0}$ be as before and let $\widetilde {\mathbf b}_{1,0}$ be the metric induced by $\widetilde {\mathbf g}_{1,0}$. It is well-defined because $\widetilde {\mathbf g}_{1,0}$ is also left-invariant with respect to $K$. Denote by $d_{\mathbf b_{1,0}}$ and $d_{\widetilde {\mathbf b}_{1,0}}$  the corresponding distance functions. Then $d_{\widetilde {\mathbf b}_{1,0}}$  also vanishes. Indeed, for an element $\phi \in \Diff S^1$, which is not a rotation, write $\pi(\phi) = f \neq \id_{\unitD}$. Then, if $\rho(\theta) = \theta+ s$, it follows that
$$\pi(\rho \circ \phi) = \widetilde f, \qquad \widetilde f(z) = e^{is} f(z e^{-is}),$$
which is different from $f$ by our previous assumptions, but $d_{\widetilde {\mathbf b}_{1,0}}(f, \widetilde f) = 0$ since $d_{\widetilde {\mathbf g}_{1,0}}(\phi, \rho \circ \phi) = 0$. However, this argument cannot be used to show that $d_{\mathbf b_{1,0}}$ vanishes because $\pi(\phi \circ \rho) = \pi(\phi)$ always, and so this remains an open question.

We remark that if $d_{\mathbf b_{1,0}}$ does not vanish, then neither will the Carnot-Carat\'eodory distance $d_{C-C}$ with respect to $\mathbf h_{1,0}$ by the obvious inequality $d_{C-C}(\phi_1, \phi_2) \geq d_{\mathbf b_{1,0}}(\pi(\phi_1), \pi(\phi_2))$.
\end{remark}

\subsection{Subgroups of $\wDiff S^1$}\label{sectionextra}
We describe some special subgroups of $\wDiff S^1$, that will be used to prove the sub-Riemannian controllability for the groups $\Diff S^1$ and $\calG_{\alpha\beta}$ in the classes of $\calH$- and $\calE$-horizontal curves respectively.

We start from describing subalgebras of $\Vect S^1$. For each $n \in \integer$, let us define
$$p_n = \cos n \theta \, \partial_\theta, \qquad k_n = \sin n\theta \, \partial_\theta.$$
The  Lie brackets are given by
\begin{equation}\label{k}
\left[k_m , k_n \right] = \tfrac{m+n}{2} k_{m-n} + \tfrac{m-n}{2} k_{m+n},
\end{equation}
\begin{equation}\label{p}
\left[p_m , p_n\right] = - \tfrac{m+n}{2} k_{m-n} - \tfrac{m-n}{2} k_{m+n},
\end{equation}
\begin{equation}\label{kp}
\left[p_m, k_n \right] = -\tfrac{m+n}{2} p_{m-n} + \tfrac{m-n}{2} p_{n+m}.
\end{equation}

It is easy to see from~(\ref{k}--\ref{kp}) that $\Lieh_n = \spn \{p_0 , p_n, k_n \}$ are subalgebras of $\Vect S^1$, and that $\Lieh_n$ is isomorphic to $\su(1,1)$ for each $n$. We need to construct an explicit exponentiation from $\Lieh_n$ to the subgroup $H_n$ of $\Diff S^1$.  We start from describing the simply connected group and its universal cover, corresponding to $\su(1,1)$.

\subsubsection{The universal cover of  $\SU(1,1)$}
The Lie group $\SU(1,1)$ consists of $2\times 2$ complex matrices
$$\left(\begin{array}{cc} z_1 & z_2 \\ \bar{z}_2 & \bar{z}_1 \end{array}\right), \qquad |z_1|^2- |z_2|^2= 1.$$
Its linear Lie algebra $\su(1,1)$ has the basis given by
$$X = \frac{1}{2} \left(\begin{array}{cc} 0 & i \\ -i & 0 \end{array}\right), \quad
Y = \frac{1}{2} \left(\begin{array}{cc} 0 & -1 \\ -1 & 0 \end{array}\right), \quad
Z = \frac{1}{2} \left(\begin{array}{cc} -i & 0 \\ 0 & i \end{array}\right).$$

Denote by $\widetilde{\SU}(1,1)$ the universal covering group of $\SU(1,1)$. The universal cover  $\widetilde{\SU}(1,1)$ can be represented as $\real \times \comp$ endowed with the group operation
$$(s_1, w_1)\cdot (s_2, w_2) = (s_3, w_3),\qquad (s_j, w_j) \in \real \times \comp, \quad j=1,2,3,$$
where
\begin{align*}
s_3 \, &= s_1 + s_2 + \Arg\left(\sqrt{(|w_1|^2 +1)(|w_2|^2 +1)} + \bar{w}_1 w_2 e^{-i(s_1+s_2)} \right), \\
w_3 \, &= w_2 e^{-is_1} \sqrt{|w_1|^2 +1} + w_1 e^{is_2} \sqrt{|w_2|^2 +1}.
\end{align*}
In these coordinates, the covering  homomorphism from $\widetilde{\SU}(1,1)$ to $\SU(1,1)$ is given by
$$(s, w) \mapsto \left(\begin{array}{cc} e^{-is} \sqrt{|w|^2 +1} & w \\ \bar{w} & e^{is} \sqrt{|w|^2 +1} \end{array}\right).$$

Let us introduce necessary notations in order to describe the exponential map from $\su(1,1)$ to $\SU(1,1)$ and $\widetilde{\SU}(1,1)$.
For a vector $a = (a_1, a_2, a_3) \in \mathbb{R}^3$, we define its Lorentzian norm as
$\check{a} = a_1^2 + a_2^2 - a_3^2$.
We define the following  functions from $\real$ to $\real$
$$\displaystyle \scrC_a(t) = \left\{\begin{array}{ll} \cosh \left(\sqrt{\check{a}} t\right) & \text{if } \check{a} \geq 0, \\
\cos \left(\sqrt{-\check{a}} t\right) & \text{if } \check{a} < 0, \end{array} \right.
\qquad
\scrS_a(t) = \left\{\begin{array}{ll} \displaystyle \tfrac{\sinh \left(\sqrt{\check{a}} t\right)}{\sqrt{\check{a}}} & \text{if } \check{a} > 0, \\
t & \text{if } \check{a} = 0, \\ \displaystyle
\tfrac{\sin \left(\sqrt{-\check{a}} t\right)}{\sqrt{-\check{a}}} & \text{if } \check{a} < 0. \end{array} \right.$$
Continue and define $\scrT_a(t)$ for $\check{a} \geq 0$,
$$\scrT_a(t) = \left\{ \begin{array}{ll}  \tan^{-1}\left(\displaystyle\tfrac{a_3}{\sqrt{\check{a}}} \tanh\left(\sqrt{\check{a}} t\right)\right) & \text{if } \check{a} >0, \\ \displaystyle
\tan^{-1} a_3 t & \text{if } \check{a} = 0,\end{array}\right.$$
while the formula for $\check{a}< 0$ is given by
$$\scrT_a(t) = \left\{ \begin{array}{ll} \tan^{-1}\left(\displaystyle\tfrac{a_3}{\sqrt{-\check{a}}} \tan\left(\sqrt{-\check{a}} t
\right)\right)  + \pi n_a(t)& \text{if } t\sqrt{-\check{a}} \neq \tfrac{\pi}2 \text{ mod } \pi, \\
\sgn(a_3) t \sqrt{- \check{a}} & \text{if } t\sqrt{-\check{a}} = \tfrac{\pi}2 \text{ mod } \pi, \end{array}\right.$$
where
$$n_a(t) = \sgn(a_3) \left\lceil \tfrac{t \sqrt{- \check{a}}}{\pi} - \tfrac{1}{2} \right\rceil ,$$
$$\sgn(t) = \left\{ \begin{array}{ll} 1 & \text{if } t > 0 \\ 0 & \text{if } t = 0 \\ -1 & \text{if } t < 0 \end{array} \right. ,
\qquad \lceil t \rceil = \min\{ j \in \integer \, : \, t \leq j \}.$$
Then the exponential map to $\widetilde{\SU}(1,1)$ is given as
\begin{equation}\label{exptildeSU11}
\exp_{\wSU(1,1)}\Big(t(a_1X+a_2Y +a_3 Z)\Big) = \Big(\scrT_a(\tfrac{t}{2}), i(a_1 + i a_2) \scrS_a(\tfrac{t}{2}) \Big),
\end{equation}
and the exponential map to $\SU(1,1)$ is written as
\begin{equation} \label{expSU11}
\exp_{SU(1,1)}\Big(t(a_1X + a_2Y + a_3 Z)\Big) = \left(\begin{array}{cc} \scrC_a(\tfrac{t}{2}) - ia_3 \scrS_a(\tfrac{t}{2})
& i(a_1 + i a_2) \scrS_a(\tfrac{t}{2}) \\ -i(a_1 - i a_2) \scrS_a(\tfrac{t}{2}) &
\scrC_a(\tfrac{t}{2}) + ia_3 \scrS_a(\tfrac{t}{2}) \end{array}\right).
\end{equation}
See details in~\cite{GV}.

\subsubsection{Embedding of $\wSU(1,1)$ into $\wDiff S^1$} \label{secHn}
We want to find an explicit expression for the subgroups corresponding to the sub-algebras $\Lieh_n$.
\begin{proposition}
The subgroup $\widetilde H_n$ of $\wDiff S^1$, is the group of diffeomorphisms of the form
$$\phi(\theta) = \theta + \tfrac{2}{n} s + \tfrac{2}{n} \Arg\left(\sqrt{|w|^2 + 1} -i \bar{w}
e^{-i(n\theta +s)}\right), \qquad s \in \real,\ \  w \in \comp.$$
\end{proposition}
\begin{proof}
For any positive integer $n$, define the mapping
$$f_n: \wSU(1,1) \to \wDiff S^1,$$
by
$$f_n(s,w)(\theta) = \theta + \tfrac{2}{n} s + \tfrac{2}{n} \Arg\left(\sqrt{|w|^2 + 1} -i \bar{w}
e^{-i(n\theta +s)}\right).$$
We want to show that $f_n$ is an injective group homomorphism.

The homomorphism property follows from computation of $f_n(s_1, w_1) \circ f_n(s_2, w_2)$. Let $(s_1,w_1)\cdot (s_2,w_2) = (s_3, w_3)$. For a fixed value $\theta$, define
$$\vartheta := f_n(s_2, w_2)(\theta).$$

Then
\begin{align*}
& f_n(s_1, w_1) \circ f_n(s_2, w_2)(\theta) \\
= & \,\, \vartheta  + \tfrac{2}{n}s_1   + \tfrac{2}{n}\Arg\left(\sqrt{|w_1|^2 +1} - i \bar{w}_1e^{-i(n \vartheta + s_1)} \right) \\
= & \,\, \theta  + \tfrac{2}{n}(s_1 + s_2) + \tfrac{2}{n}\Arg\left(\sqrt{|w_2|^2 +1} - i w_2 e^{-i(n\theta + s_2)} \right)  \\
& + \tfrac{2}{n}\Arg\left(\sqrt{|w_1|^2 +1} - i \bar{w}_1e^{-i(s_1+ n\theta + 2 s_2)}
\frac{\sqrt{|w_2|^2 +1} + i w_2 e^{i(n\theta + s_2)}} {\sqrt{|w_2|^2 +1} - i \bar{w}_2 e^{-i(n\theta + s_2)}}\right)
\end{align*}
\begin{align*}
= & \,\, \theta  + \tfrac{2}{n}(s_1 + s_2) + \tfrac{2}{n}\Arg\Big(\sqrt{(|w_1|^2 +1)(|w_2|^2 +1)}
\\ & - i\bar{w}_2 e^{-i(n\theta + s_2)} \sqrt{|w_1|^2 +1}
- i \bar{w}_1 e^{-i(n\theta - 2s_2 - s_1)} \sqrt{|w_2|^2 +1} + \bar{w}_1 w_2 e^{-i(s_1+s_2)}\Big) \\
= & \,\, \theta  + \tfrac{2}{n}(s_1 + s_2) + \tfrac{2}{n}\Arg\left(e^{i(s_3-s_1 - s_2)} \sqrt{|w_3|^2 +1}
- i\bar{w}_3 e^{-i(n\theta + s_1 +s_2)}\right)
\end{align*}
\begin{align*}
= & \,\, \theta  + \tfrac{2}{n}(s_1 + s_2) + \tfrac{2}{n} \Arg(e^{i(s_3-s_1 - s_2)})
+ \tfrac{2}{n}\Arg\left(\sqrt{|w_3|^2 +1} - i\bar{w}_3 e^{-i(n\theta + s_3)}\right) \\
= & \,\, \theta  + \tfrac{2}{n}s_3 + \tfrac{2}{n}\Arg\left(\sqrt{|w_3|^2 +1} - i\bar{w}_3 e^{-i(n\theta + s_3)}\right) \\
= & f_n(s_3, w_3).
\end{align*}
To show injectivity of $f_n$, assume that $f(s,w) = \id$. This implies that for any $\theta$,
\begin{equation} \label{eqinjective}\Arg\left(\sqrt{|w|^2 + 1} -i \bar{w}e^{-i(n\theta +s)}\right) = - s.\end{equation}
However, the left side of \eqref{eqinjective} is  constant only if $w =0$, which, in its turn, implies $s=0$. Hence, the kernel of $f_n$ is trivial.

To complete the proof, we need to show that $\Lieh_n$ is indeed the image of $d_{id} f_{n}$.
But this follows from the computations
\begin{equation} \label{eqpartialthree} \partial_t |_{t= 0} f(t,0)(\theta) = \tfrac{2}{n}, \quad \partial_t |_{t= 0} f(0,t)(\theta) = \tfrac{2}{n} \cos n\theta,
\quad \partial_t |_{t= 0} f(0,it)(\theta) = -\tfrac{2}{n} \sin n\theta.\end{equation}
\end{proof}

\begin{remark} Let us give a couple of observations.
\begin{itemize}
\item[1.]{
Since the derivative of the curves $t \mapsto (t,0), t \mapsto (0,t)$ and $t \mapsto (0,it)$ at the identity in $\wSU(1,1)$ can be identified with $2Z, -2Y$ and $2X$, respectively, relations~\eqref{eqpartialthree} yield that
$$\exp_{\wDiff S^1}\Big(t(a_1k_n + a_2 p_n + a_3 p_0)\Big) = f_n\left(\exp_{\wSU(1,1)} \Big(nt(-a_1 X - a_2 Y + a_3 Z) \Big) \right).$$}
\item[2.]{To obtain the corresponding subgroups in $\Diff S^1$ one needs to add a (mod $2\pi$) at the end.}
\end{itemize}
\end{remark}

\subsection{Controllability}

We will finish this section by addressing the question of controllability. Notice first that although the brackets of $\Vect_0 S^1$ generate the whole algebra $\Vect S^1$, it is not a $C^\infty(S^1)$-sub-module so we cannot apply Theorem~\ref{theorem:AC}. Instead we must use the invariance under the group action of the horizontal sub-bundles to show that we indeed can connect every pair of points.
\begin{theorem} The following is true.
\begin{itemize}
\item[(a)] Let $\calH$ be a choice of horizontal sub-bundle on $\Diff S^1$ or $\wDiff S^1$ defined as in Section~\ref{sec:HCRC}. Then any pair of points can be connected by an $\calH$-horizontal curve.
\item[(b)] Let $\calE$ be a choice of horizontal sub-bundle on $\calG_{\mu\nu}$ defined as in Section~\ref{sec:HCRC}. Then any two points on $\calG_{\mu\nu}$ can be connected by an $\calE$-horizontal curve.
\end{itemize}
\end{theorem}

\begin{proof}
To prove (a), it is sufficient to show that any two points in $\wDiff S^1$ can be connected by an $\calH$-horizontal curve. Applying  Proposition~\ref{groupcontrollability}, we only need to verify that the unit of the group $\id$ can be connected with any element in $\widetilde K$ by an $\calH$-horizontal curve. 
The subgroup $\widetilde K$ is contained in $\widetilde H_n$ for any $n$, where $\widetilde H_n$ are as described in Section~\ref{sectionextra}. In particular, $\widetilde K$ can be considered as a subgroup of~$H_1$. Any $\calH$-horizontal curve in $H_1$, has the left logarithmic derivative in $\Lieh_1 \cap \Vect_0 S^1 = \spn\{k_1, p_1\}$. Since $[p_1,k_1] = p_0$, the horizontal distribution $\calH$ restricted to $H_1$ is bracket generating. The group $H_1$ is finite-dimensional, therefore, we can apply the Rashevski{\u\i}-Chow theorem to conclude that every point in $H_1$, including points in $\widetilde K$, can be reached by an $\calH$-horizontal curve.

To prove (b), we need to show that any point in $\widehat K=\{\theta\mapsto (\theta+b_0,b)\in\calG_{\mu\nu}\}$ can be connected to $(\id,0)$ by an $\calE$-horizontal curve. A bit more care needs to be taken in this case. Consider subgroups
$$\widehat H_n = \left\{(\phi,a) \in \calG_{\mu\nu} \, : \, \phi \in H_n,\ a\in\mathbb R \right\}.$$
which has the Lie algebras
$$\widehat{\Lieh}_n = \spn \left\{(p_0, 0), (p_n, 0), (k_n, 0), (0,1)\right\},$$
The Lie algebras $\widehat{\Lieh}_n$ have special sub-algebras
$$\widehat{\mathfrak t}_n = \spn \left\{(p_0, n^2 \nu - \mu), (p_n, 0), (k_n, 0)\right\}.$$
Denote the corresponding subgroups by $\widehat T_n$. In the contrast to $\widehat H_n$, the distribution $\calE$ restricted any subgroup $\widehat T_n$ is bracket generating, and so all elements in such a subgroup $\widehat T_n$ can be reached by an $\calE$-horizontal curve. It is clear that
$$\widehat T_n \cap \widehat K = \left\{\big(\theta \mapsto \theta + r, r(n^2 \nu - \mu)\big) \, : \, r \in \real\right\}.$$
Since $\widehat K$ is isomorphic to $\real^2$ as a group and $\nu \neq 0$, we know that for any $g \in \widehat K$, there are unique elements $g_j \in \widehat T_j \cap \widehat K, j = 1, 2$ so that $g= g_1 \cdot g_2 = g_2 \cdot g_1$. Denote by $c_1$ and $c_2$ curves that connect $\id$ with $g_1$ and $g_2$ respectively. We can reach $g$ first by  following the curve $c_2$ and and then continuing by $\ell_{g_2} c_1$. This finishes the proof.
\end{proof}


\end{document}